\documentclass[12pt,a4paper]{amsart} 
\usepackage{amsmath,amsfonts,amssymb,amsthm,amscd,slashed}
\usepackage{epsfig}
\pdfoutput=1

\usepackage{tikz}

\usepackage{xypic}

\usepackage{hyperref}
\usepackage{JK}                

\newtheorem{thm}{Theorem}
\newtheorem{corl}[thm]{Corollary}
\newtheorem{lma}[thm]{Lemma} \newtheorem{prop}[thm]{Proposition}
\newtheorem{defn}[thm]{Definition}

\newtheorem{assu}{Assumption}
\newtheorem{rem}[thm]{Remark}

\newcommand{\red}{\color{red}}

\def\1{\mathbb{I}}

\def\C{\mathbb{C}}

\DeclareMathOperator{\dom}{Dom}

\def\E{\mathcal{E}}

\DeclareMathOperator{\im}{Im}

\def\R{\mathbb{R}}

\def\T{\mathbb{T}}

\def\tilde{\widetilde}

\title{On a theorem of Kucerovsky for half-closed chains}
\author{Jens Kaad and Walter D. van Suijlekom}

\address{Department of Mathematics and Computer Science, Syddansk Universitet, Campusvej 55, 5230, Odense M, Denmark}
\email{jenskaad@hotmail.com}
\address{Institute for Mathematics, Astrophysics and Particle Physics, Radboud University Nijmegen, Heyendaalseweg 135, 6525 AJ Nijmegen, The Netherlands}
\email{waltervs@math.ru.nl}
\date{\today}


\begin{document}
\maketitle

\begin{abstract}
Kucerovsky's theorem provides a method for recognizing the interior Kasparov product of selfadjoint unbounded cycles. In this paper we extend Kucerovsky's theorem to the non-selfadjoint setting by replacing unbounded Kasparov modules with Hilsum's half-closed chains. On our way we show that any half-closed chain gives rise to a multitude of twisted selfadjoint unbounded cycles via a localization procedure. These unbounded modular cycles allow us to provide verifiable criteria avoiding any reference to domains of adjoints of symmetric unbounded operators.
\end{abstract}

\section{Introduction}
In recent years a lot of attention has been given to the non-unital framework for noncommutative geometry, where the absence of a unit is interpreted as a non-compactness condition on the underlying noncommutative space, \cite{Con:NCG,Lat:QLM,CGRS:ILN,MeRe:NMU}. For a more detailed analysis of the non-compact setting it is important to distinguish between the complete and the non-complete case, \cite{MeRe:NMU}. Whereas the complete case is still modelled by a (non-unital) spectral triple or more generally an unbounded Kasparov module, the lack of completeness leads to the non-selfadjointness of symmetric differential operators. A noncommutative geometric framework that captures the non-complete setting is provided by Hilsum's notion of a half-closed chain, where the selfadjointness condition on the unbounded operator is replaced by a more flexible symmetry condition, \cite{Hil:BIK}. This framework is supported by results of Baum, Douglas, Taylor and Hilsum showing that \emph{any} first-order symmetric elliptic differential operator on \emph{any} Riemannian manifold gives rise to a half-closed chain, \cite{BDT:CRA,Hil:BIK}.

Unbounded Kasparov modules give rises to classes in Kasparov's KK-theory via the Baaj-Julg bounded transform and this result has been extended by Hilsum to cover half-closed chains, \cite{BaJu:TBK,Hil:BIK}. This transform contains information about the algebraic topology of the original geometric situation described by a half-closed chain.

The main structural property of Kasparov's KK-theory is the interior Kasparov product, \cite{Kas:OFE}:
\[
\hot_B : KK(A,B) \ti KK(B,C) \to KK(A,C) \, .
\]
The interior Kasparov product is however not explicitly constructed and it is therefore important to develop tools for computing the interior Kasparov product of two given Kasparov modules. Given three classes in KK-theory, Connes and Skandalis developed suitable conditions for verifying whether one of these three classes factorizes as an interior Kasparov product of the remaining two classes, \cite{CoSk:LIF}.

The conditions of Connes and Skandalis were translated to the unbounded setting by Kucerovsky, \cite{Kuc:KUM}. Thus, given three unbounded Kasparov modules, Kucerovsky's theorem provides criteria for verifying that one of these unbounded Kasparov modules factorizes as an unbounded Kasparov product of the remaining two unbounded Kasparov modules. In many cases, the conditions are easier to verify directly at the unbounded level, using Kucerovsky's theorem, instead of first applying the bounded transform and then relying on the results of Connes and Skandalis. Indeed, in the unbounded setting we are usually working with first-order differential operators whereas their bounded transforms are zeroth-order pseudo-differential operators involving a square root of the resolvent.

In this paper we extend Kucerovsky's theorem to cover the non-complete setting, where the unbounded Kasparov modules are replaced by half-closed chains. The main challenge in carrying out such a task is that the domain of the adjoint of a symmetric unbounded operator can be difficult to describe. The original proof of Kucerovsky does therefore not translate to the non-selfadjoint setting as the correct conditions have to be formulated without any reference to maximal domains of symmetric unbounded operators. 

The main technique that we apply is a localization procedure relating to the work of the first author in \cite{Kaa:UKM,Kaa:DAH}. This procedure allows us to pass from a symmetric regular unbounded operator $D$ to an {\it essentially selfadjoint} regular unbounded operator of the form $x D x^*$ for an appropriate bounded adjointable operator $x$. In the case where $D$ is a Dirac operator, the localization corresponds to a combination of two operations: restricting all data to an open subset and passing from the non-complete Riemannian metric on this open subset to a conformally equivalent but complete Riemannian metric. The size of the open neighborhood and the relevant conformal factor are both determined by the positive function $xx^*$. 

In particular, our technique allows us to construct a multitude of unbounded modular cycles out of a given half-closed chain. We interpret this localization procedure in terms of the unbounded Kasparov product by the module generated by the localizing element $x$. In this way, we may work with selfadjoint unbounded operators and hence eliminate the difficulties relating to the description of maximal domains. On the other hand, the ``conformal factor'' $(xx^*)^{-2}$ produces a twist of the commutator condition and this twist is described by the modular automorphism $\sigma(\cd ) = (xx^*)( \cdot ) (xx^*)^{-1}$. We refer to Connes and Moscovici for further discussion of this issue in the case where $x$ is positive and invertible, see \cite{CM08}. 

The present paper is motivated by the geometric setting of a proper Riemannian submersion of spin$^{c}$-manifolds, and the criteria that we develop here have already been applied in \cite{KaSu:FDA} to obtain factorization results involving the corresponding fundamental classes in KK-theory.

Our results may also be of importance for the further development of the unbounded Kasparov product as initiated by Connes in \cite{Con:GMF} and developed further by Mesland and others in \cite{Mes:UCN,KaLe:SFU,BMS13,Kaa:UKM,MeRe:NMU,Kaa:MEU}.

The structure of this paper is as follows: In Section \ref{s:half} and Section \ref{s:modular} we review the concept of a half-closed chain and of an unbounded modular cycle. In Section \ref{s:locaregu}, Section \ref{s:locahalf} and Section \ref{s:unbkas} we prove our results on the localization procedure and investigate how it relates to the Kasparov product. In Section \ref{s:kuce} we prove Kucerovsky's theorem for half-closed chains.

\subsection*{Acknowledgements}
We would like to thank Georges Skandalis for a highly stimulating remark concerning the ``locality'' of Kucerovsky's theorem.

This work also benefited from various conversations with Magnus Goffeng and Bram Mesland.

The authors gratefully acknowledge the Syddansk Universitet Odense and the Radboud University Nijmegen for their financial support in facilitating this collaboration. 

During the initial stages of this research project the first author was supported by the Radboud excellence fellowship.
%
%

The first author was partially supported by the DFF-Research Project 2 ``Automorphisms and Invariants of Operator Algebras'', no. 7014-00145B and by the Villum Foundation (grant 7423).

The second author was partially supported by NWO under VIDI-grant 016.133.326.
%

\section{Half-closed chains}\label{s:half}
Let us fix two $\si$-unital $C^*$-algebras $A$ and $B$.

Let $E$ be a countably generated Hilbert $C^*$-module over $B$. We recall that a closed (densely defined) unbounded operator $D : \Tex{Dom}(D) \to E$ is said to be {\it regular} when it has a densely defined adjoint $D^* : \Tex{Dom}(D^*) \to E$ and when $1 + D^* D : \Tex{Dom}(D^* D) \to E$ has dense range. It follows from this definition that $1 + D^* D : \Tex{Dom}(D^* D) \to E$ is in fact densely defined and surjective, \cite[Lemma 9.1]{Lan:HMT}. In particular we have a bounded adjointable inverse $(1 + D^* D)^{-1} : E \to E$.

For two countably generated Hilbert $C^*$-modules $E$ and $F$ over $B$, we let $\mathbb{L}(E,F)$ and $\mathbb{K}(E,F)$ denote the bounded adjointable operators from $E$ to $F$ and the compact operators from $E$ to $F$, respectively. When $E = F$ we put $\mathbb{L}(E) := \mathbb{L}(E,F)$ and $\mathbb{K}(E) := \mathbb{K}(E,F)$. We let $\| \cd \|_\infty : \mathbb{L}(E,F) \to [0,\infty)$ denote the operator norm.

The following definition is due to Hilsum, \cite[Section 3]{Hil:BIK}:

\begin{defn}
A \emph{half-closed chain} from $A$ to $B$ is a triple $(\sA,E,D)$, where $\sA \subseteq A$ is a norm-dense $*$-subalgebra, $E$ is a countably generated $C^*$-correspondence from $A$ to $B$ and $D : \Tex{Dom}(D) \to E$ is a closed, symmetric and regular unbounded operator such that
\begin{enumerate}
\item $a \cd (1 + D^* D)^{-1}$ is a compact operator on $E$ for all $a \in A$;
\item $a\big( \Tex{Dom}(D^*) \big) \su \Tex{Dom}(D)$ for all $a \in \sA$;
\item $[D,a] : \Tex{Dom}(D) \to E$ extends to a bounded operator $d(a) : E \to E$ for all $a \in \sA$.
\end{enumerate}

A half-closed chain $(\sA,E,D)$ from $A$ to $B$ is said to be \emph{even} when $E$ comes equipped with a $\zz/2\zz$-grading operator $\ga : E \to E$ ($\ga = \ga^*$, $\ga^2 = 1$), such that $[a,\ga] = 0$ for all $a \in A$ and $D \ga = - \ga D$.

A half-closed chain which is not even is said to be \emph{odd}.
\end{defn}

Let $(\sA,E,D)$ be a half-closed chain from $A$ to $B$. A few observations are in place: 
\begin{enumerate}
\item $d(a) : E \to E$, $a \in \sA$, is automatically adjointable with $d(a)^* = - d(a^*)$.
\item The difference
\[
D a - a D^* : \Tex{Dom}(D^*) \to E \q a \in \sA
\]
extends to the bounded adjointable operator $d(a) : E \to E$.
\item $a \cd (1 + D D^*)^{-1} \in \mathbb{K}(E)$ for all $a \in A$. (Remark that $D^*$ is automatically regular by \cite[Proposition 9.5]{Lan:HMT}).
\end{enumerate}

We recall that a \emph{Kasparov module} from $A$ to $B$ is a pair $(E,F)$ where $E$ is a countably generated $C^*$-correspondence from $A$ to $B$ and $F : E \to E$ is a bounded adjointable operator such that
\[
a \cd ( F - F^*) \, , \, \, a \cd (F^2 - 1) \, , \, \, [F,a] \in \mathbb{K}(E) \, ,
\]
for all $a \in A$. A Kasparov module $(E,F)$ from $A$ to $B$ is \emph{even} when it comes equipped with a $\zz/2\zz$-grading operator $\ga : E \to E$ such that $[ a,\ga] = 0$ for all $a \in A$ and $F \ga + \ga F = 0$. Otherwise we say that $(E,F)$ is \emph{odd}.

For an unbounded regular operator $D : \Tex{Dom}(D) \to E$ we let $F_D := D(1 + D^* D)^{-1/2} \in \mathbb{L}(E)$ denote the \emph{bounded transform} of $D$. We have that $F_D^* = F_{D^*} = D^*(1 + D D^*)^{-1/2}$.

The next result creates the main link between half-closed chains and Kasparov modules. This result is due to Hilsum, \cite{Hil:BIK}, and it generalizes the corresponding result of Baaj and Julg for unbounded Kasparov modules, \cite{BaJu:TBK}. Remark however that the condition $[F_D,a] \in \mathbb{K}(E)$, $a \in A$, is for some reason left unproved in \cite[Theorem 3.2]{Hil:BIK}. We therefore give a full proof of this commutator condition here:

\begin{thm}\label{t:kasphalf}
Suppose that $(\sA,E,D)$ is a half-closed chain from $A$ to $B$. Then $(E,F_D)$ is a Kasparov module from $A$ to $B$ of the same parity as $(\sA,E,D)$ and with the same $\zz/2\zz$-grading operator $\ga : E \to E$ in the even case.
\end{thm}
\begin{proof}
We have to show that $[F_D,a] \in \mathbb{K}(E)$ for all $a \in A$. Since the $*$-algebra $\sA \su A$ is dense in $C^*$-norm and since the $C^*$-algebra $\mathbb{K}(E) \su \mathbb{L}(E)$ is closed in operator norm it suffices to show that $[F_D,a] \cd b \in \mathbb{K}(E)$ for all $a,b \in \sA$.

We recall that
\[
(1 + D^* D)^{-1/2} = \frac{1}{\pi} \int_0^\infty \la^{-1/2} (1 + \la + D^* D)^{-1} \, d\la \, ,
\]
where the integral converges absolutely in operator norm and where the integrand is continuous in operator norm. Remark here that $\| (1 + \la + D^* D)^{-1} \|_\infty \leq (1 + \la)^{-1}$ for all $\la \geq 0$. 

For $a \in \sA$ and $\la \geq 0$ we then compute that
\[
\begin{split}
& \big[  D (1 + \la + D^* D)^{-1} , a \big] \\
& \q = - D D^* (1 + \la + D D^*)^{-1} d(a) (1 + \la + D^* D)^{-1} \\ 
& \qq - D (1 + \la + D^* D)^{-1} d(a) D(1 + \la + D^* D)^{-1} \\
& \qqq + d(a)(1 + \la + D^* D)^{-1} \, .
\end{split}
\]
In particular, it holds for each $a,b \in \sA$ that the map
\[
M : (0,\infty) \to \mathbb{L}(E) \q M(\la) := \la^{-1/2} [  D (1 + \la + D^* D)^{-1} , a ] b
\]
is continuous in operator norm and that $M(\la) \in \mathbb{K}(E)$ for all $\la \in (0,\infty)$. Moreover, we have the estimate
\[
\| M(\la) \|_\infty \leq \la^{-1/2} \cd \| d(a) \| \cd 3 \cd (1 + \la)^{-1}  \, ,
\]
for all $\la > 0$. We may thus conclude that
\[
[F_D,a] b = \frac{1}{\pi} \int_0^\infty M(\la) \, d\la \in \mathbb{K}(E) \, ,
\]
for all $a,b \in \sA$. This proves the theorem.
\end{proof}




\section{Unbounded modular cycles}\label{s:modular}
Let us fix $\si$-unital $C^*$-algebras $A$ and $B$ together with a dense $*$-subalgebra $\sA \su A$.

The following definition is from \cite[Section 3]{Kaa:UKM}:

\begin{defn}\label{d:modular}
An \emph{unbounded modular cycle} from $\sA$ to $B$ is a triple $(E,D,\De)$ where $E$ is a countably generated $C^*$-correspondence from $A$ to $B$, $D : \Tex{Dom}(D) \to E$ is an unbounded selfadjoint and regular operator, and $\De : E \to E$ is a bounded positive and selfadjoint operator with norm-dense image such that
\begin{enumerate}
\item $a (i + D)^{-1} : E \to E$ is a compact operator for all $a \in A$;
\item $(a + \la) \Delta$ has $\Tex{Dom}(D) \subseteq E$ as an invariant submodule and
\[
D (a + \la) \De - \De (a + \la) D : \Tex{Dom}(D) \to E
\]
extends to a bounded adjointable operator $d_\De(a,\la) : E \to E$ for all $a \in \sA$, $\la \in \cc$.
\item The supremum
\[
\sup_{\ep > 0} \| (\De^{1/2} + \ep)^{-1} d_\De(a,\la) (\De^{1/2} + \ep)^{-1} \|_\infty
\]
is finite for all $a \in \sA$, $\la \in \cc$.
\item The sequence $\{ \De (\De + 1/n)^{-1} a \}$ converges in operator norm to $a$ for all $a \in A$.
\end{enumerate}

An unbounded modular cycle is \emph{even} when $E$ comes equipped with a $\zz/2\zz$-grading operator $\ga : E \to E$ ($\ga = \ga^*$, $\ga^2 = 1$), such that $[a,\ga] = 0$ for all $a \in A$ and $D \ga = - \ga D$.

An unbounded modular cycle is \emph{odd} when it is not even.
\end{defn}

\begin{rem}
Note that if $\Delta$ has a bounded inverse then (3) and (4) are automatic. If, in addition, $A$ is unital, $\Delta, \Delta^{-1} \in \sA$ and $B=\C$ then the modular cycle $(E,D,\Delta)$ defines a twisted spectral triple in the sense of \cite{CM08}, with the twisting automorphism $\sigma : \sA \to \sA$ given by $\sigma(a) = \Delta a \Delta^{-1}$ for all $a \in \sA$. 
\end{rem}

\begin{rem}\label{r:cbb}
In \cite{Kaa:UKM} it is assumed that $\sA$ is equipped with a fixed operator space norm $\| \cd \|_1 : M_n(\sA) \to [0,\infty)$, $n \in \nn$, such that the inclusion $\sA \to A$ is completely bounded. In the above definition it is then required that the supremum in $(3)$ is completely bounded in the sense that
\[
\sup_{\ep > 0} \| (\De^{1/2} + \ep)^{-1} d_\De(a,0) (\De^{1/2} + \ep)^{-1} \|_\infty \leq C \cd \| a \|_1
\]
for all $a \in M_n(\sA)$, $n \in \nn$ (thus, the constant $C$ is independent of the size of the matrices). This structure is relevant for the construction of the unbounded Kasparov product, but will not play a role in the present text.
\end{rem}

As in the case of half-closed chains, each unbounded modular cycle represents an explicit class in KK-theory. This result can be found as \cite[Theorem 9.1]{Kaa:UKM}. We state it here for the convenience of the reader. We recall that $F_D := D(1 + D^2)^{-1/2}$ denotes the bounded transform of $D : \Tex{Dom}(D) \to E$ (but now $D$ is selfadjoint and regular).

\begin{thm}\label{t:modkas}
Suppose that $(E,D,\De)$ is an unbounded modular cycle from $\sA$ to $B$. Then $(E, F_D)$ is a Kasparov module from $A$ to $B$ of the same parity as $(E,D,\De)$ and the same $\zz/2\zz$-grading operator $\ga : E \to E$ in the even case.
\end{thm}





\section{Localization of regular unbounded operators}\label{s:locaregu}
Let $E$ be a countably generated Hilbert $C^*$-module over a $\si$-unital $C^*$-algebra $B$ and let $D : \Tex{Dom}(D) \to E$ be a closed, symmetric and regular unbounded operator.

\begin{assu}\label{a:self}
It will be assumed that $\De : E \to E$ is a bounded selfadjoint operator such that
\begin{enumerate}
\item $\De\big( \Tex{Dom}(D^*) \big) \su \Tex{Dom}(D)$ ;
\item $D \De - \De D : \Tex{Dom}(D) \to E$ extends to a bounded operator $d(\De) : E \to E$.
\end{enumerate}
\end{assu}

Remark that it follows by the above assumption and the inclusion $D \su D^*$ that
\[
D \De - \De D^* : \Tex{Dom}(D^*) \to E
\]
also has $d(\De) : E \to E$ as a bounded extension. Moreover, $d(\De) : E \to E$ is automatically adjointable with $d(\De)^* = - d(\De)$.

Before proving our first result, we notice that $D\De : \Tex{Dom}(D\De) \to E$ is a closed unbounded operator on the domain
\[
\Tex{Dom}(D\De) := \big\{ \xi \in E \mid \De(\xi) \in \Tex{Dom}(D) \big\} \, .
\]
A similar remark holds for $D^* \De : \Tex{Dom}(D^* \De) \to E$.

\begin{prop}\label{p:reg}
Suppose that the conditions in Assumption \ref{a:self} hold. Then
\[
D \De = D^* \De
\]
and $D\De : \Tex{Dom}(D\De) \to E$ is a regular unbounded operator with core $\Tex{Dom}(D)$ and with
\[
(D\De)^* = D\De - d(\De) \, .
\]
In particular, we have that
\[
\Tex{Dom}( (D\De)^* ) = \Tex{Dom}(D\De) \, .
\]
\end{prop}
\begin{proof}
We first claim that the unbounded operators $D\De : \Tex{Dom}(D\De) \to E$ and $D^* \De : \Tex{Dom}(D^* \De) \to E$ are regular with cores $\Tex{Dom}(D^*)$ and $\Tex{Dom}(D)$, respectively, and with adjoints
\[
(D \De)^* = D\De - d(\De) \q \Tex{and} \q (D^* \De)^* = D^* \De - d(\De) \, .
\]

To prove this claim, we recall that $D : \Tex{Dom}(D) \to E$ is regular by assumption, and we thus have that 
\[
\ma{cc}{0 & D^* \\ D & 0} : \Tex{Dom}(D) \op \Tex{Dom}(D^*) \to E \op E
\]
is selfadjoint and regular. Moreover, we have that
\[
\ma{cc}{0 & \De \\ \De & 0} \big( \Tex{Dom}(D) \op \Tex{Dom}(D^*) \big) \su \Tex{Dom}(D) \op \Tex{Dom}(D^*)
\]
and the identities
\[
\begin{split}
\big[ \ma{cc}{0 & D^* \\ D & 0} , \ma{cc}{0 & \De \\ \De & 0} \big]
& = \ma{cc}{ D \De - \De D & 0 \\ 0 & D \De - \De D^* } \\
& = \ma{cc}{d(\De) & 0 \\ 0 & d(\De)}
\end{split}
\]
hold on $\Tex{Dom}(D) \op \Tex{Dom}(D^*)$. This means that $\ma{cc}{0 & D^* \\ D & 0}$ and $\ma{cc}{0 & \De \\ \De & 0}$ satisfy the conditions of \cite[Section 6]{Kaa:DAH} and we may conclude that
\[
\begin{split}
\ma{cc}{0 & D^* \\ D & 0} \ma{cc}{0 & \De \\ \De & 0} & = \ma{cc}{D^* \De & 0 \\ 0 & D \De}  \\
& \q : \Tex{Dom}(D^* \De) \op \Tex{Dom}(D\De) \to E \op E
\end{split}
\]
is a regular unbounded operator with
\[
\begin{split}
\ma{cc}{D^* \De & 0 \\ 0 & D \De}^* & = \ma{cc}{D^* \De & 0 \\ 0 & D \De} - \ma{cc}{d(\De) & 0 \\ 0 & d(\De)} \\
& \q : \Tex{Dom}(D^* \De) \op \Tex{Dom}(D \De ) \to E \op E
\end{split}
\]
Moreover, we know that $\ma{cc}{D^* \De & 0 \\ 0 & D \De} : \Tex{Dom}(D^* \De) \op \Tex{Dom}(D\De) \to E \op E$ has $\Tex{Dom}(D) \op \Tex{Dom}(D^*)$ as a core. This proves the claim.

To end the proof of the proposition, it now suffices to prove that $D\De = D^* \De$. To this end, we notice that
\begin{equation}\label{eq:dom}
(D^* \De)(\xi) = (D\De)(\xi) \q \Tex{for all } \xi \in \Tex{Dom}(D^*)
\end{equation}
Since $\Tex{Dom}(D) \su \Tex{Dom}(D^*)$ is a core for $D^* \De$ we obtain from Equation \eqref{eq:dom} that $D^* \De \su D\De$. Moreover, since $\Tex{Dom}(D^*)$ is a core for $D\De$ we also obtain from Equation \eqref{eq:dom} that $D\De \su D^* \De$. We conclude that $D\De = D^* \De$.
\end{proof}

\begin{assu}\label{a:loca}
It will be assumed that $x : E \to E$ is a bounded adjointable operator such that
\begin{enumerate}
\item $x\big( \Tex{Dom}(D^*) \big) \su \Tex{Dom}(D)$ and $x^* \big( \Tex{Dom}(D^*) \big) \su \Tex{Dom}(D)$;
\item $D x - x D \Tex{ and } D x^* - x^* D : \Tex{Dom}(D) \to E$ extend to bounded operators $d(x) \Tex{ and } d(x^*) : E \to E$, respectively.
\end{enumerate}
\end{assu}

As above, $d(x)$ and $d(x^*) : E \to E$ are automatically adjointable with $d(x)^* = - d(x^*)$. Moreover, $d(x)$ and $d(x^*)$ are bounded extensions of $D x - x D^* \Tex{ and } D x^* - x^* D^* : \Tex{Dom}(D^*) \to E$, respectively.

We define the \emph{localization} of $E$ (with respect to $x : E \to E$) as the Hilbert $C^*$-submodule $E_x \su E$ given by the norm-closure of the image of $x$:
\[
E_x := \Tex{cl}( \Tex{Im}(x) ) \, .
\]

We define $\De := x x^* : E \to E$.

\begin{lma}\label{l:self}
Suppose that the conditions of Assumption \ref{a:loca} are satisfied. Then the unbounded operator
\[
D \De - d(x) x^* : \Tex{Dom}(D \De) \to E
\]
is selfadjoint and regular and it has $\Tex{Dom}(D) \subseteq \Tex{Dom}(D \De)$ as a core. Moreover, we have that
\[
( D \De - d(x) x^* )(\xi) = (x D x^*)(\xi) \, ,
\]
for all $\xi \in \Tex{Dom}(D x^*) \su \Tex{Dom}(D \De)$.
\end{lma}
\begin{proof}
Clearly, $\De = x x^* : E \to E$ satisfied the conditions of Assumption \ref{a:self} and it therefore follows from Proposition \ref{p:reg} that $D \De : \Tex{Dom}(D \De) \to E$ is regular with core $\Tex{Dom}(D)$ and that
\[
(D \De)^* = D \De - d(\De) = D \De - d(x) x^* - x d(x^*) \, .
\]
Since $d(x)x^* : E \to E$ is a bounded adjointable operator, it follows by \cite[Section 2, Example 1]{Wor:UAQ} that $D \De - d(x) x^* : \Tex{Dom}(D \De) \to E$ is regular. It is moreover clear that $\Tex{Dom}(D)$ is also a core for $D \De - d(x) x^*$ and that
\[
\begin{split}
( D \De - d(x) x^* )^* & = (D \De )^* - ( d(x) x^* )^* \\ 
& = D \De - d(x) x^* - x d(x^*) + x d(x^*) = D \De - d(x) x^*  \, ,
\end{split}
\]
proving that our unbounded operator is selfadjoint as well. The final statement of the lemma is obvious.
\end{proof}

\begin{defn}\label{d:localiz}
Suppose that the conditions of Assumption \ref{a:loca} are satisfied. We define the \emph{localization} of $D : \Tex{Dom}(D) \to E$ (with respect to $x : E \to E$) as the closure of the unbounded symmetric operator
\[
x D x^* : \Tex{Dom}(D) \cap E_x \to E_x \, .
\]
The localization of $D$ is denoted by
\[
D_x : \Tex{Dom}(D_x) \to E_x \, .
\]
\end{defn}

Remark that $x\big( \Tex{Dom}(D) \big) \su \Tex{Dom}(D) \cap E_x$, implying that the localization $D_x$ is densely defined.

\begin{lma}\label{l:inv}
Suppose that the conditions of Assumption \ref{a:loca} are satisfied and let $r \in \rr$ with $|r| > \| d(x^*) x \|_\infty$ be given. Then $i r + Dx^* x : \Tex{Dom}(D x^* x) \to E$ is a bijection and the resolvent is a bounded adjointable operator $(i r + D x^* x)^{-1} : E \to E$ satisfying the relation
\begin{equation}\label{eq:resima}
(  i r +  D \De - d(x) x^* )^{-1} x = x ( i r +  D x^* x)^{-1} \, .
\end{equation}
\end{lma}
\begin{proof}
By replacing $x$ with $x^*$ in Assumption \ref{a:loca} we see from Lemma \ref{l:self} that the unbounded operator
\[
D x^* x - d(x^*) x : \Tex{Dom}(D x^* x) \to E 
\]
is selfadjoint and regular. In particular, we know that the resolvent $(i r + D x^* x - d(x^*)x)^{-1} : E \to E$ is a well-defined bounded adjointable operator. Since
\[
\| d(x^*) x (i r + D x^* x - d(x^*) x)^{-1} \|_\infty \leq \| d(x^*) x \|_\infty \cd |r|^{-1} < 1
\]
we may conclude that $ir + D x^* x : \Tex{Dom}(D x^* x) \to E$ is a bijection and that the resolvent is a bounded adjointable operator. In fact, we have that
\[
\begin{split}
(i r + D x^* x)^{-1} & = (i r + D x^* x - d(x^*) x)^{-1} \\ 
& \q \cd \big( 1 + d(x^*) x (i r + D x^* x - d(x^*) x)^{-1} \big)^{-1} \, .
\end{split}
\]
The relation in Equation Equation \eqref{eq:resima} now follows since
\[
(ir + D \De - d(x) x^*) x = (i r + x D x^* ) x = x (ir + D x^* x)
\]
on $\Tex{Dom}(D x^* x)$.
\end{proof}

\begin{prop}\label{p:loca}
Suppose that the conditions of Assumption \ref{a:loca} are satisfied. Then the localization of $D : \Tex{Dom}(D) \to E$ with respect to $x : E \to E$ is a selfadjoint and regular unbounded operator
\[
D_x : \Tex{Dom}(D_x) \to E_x \, ,
\]
with core $x(\Tex{Dom}(D)) \subseteq \Tex{Dom}(D_x)$. Moreover, we have the identity 
\begin{equation}\label{eq:reside}
( i \mu + D_x)^{-1}(\xi) = (i \mu + D \De - d(x) x^*)^{-1}(\xi) \, ,
\end{equation}
for all $\xi \in E_x$ and all $\mu \in \rr\sem \{0\}$. In particular, $E_x \su E$ is an invariant submodule for $(i \mu + D\De - d(x) x^*)^{-1} : E \to E$ for all $\mu \in \rr\sem \{0\}$.
\end{prop}
\begin{proof}
To show that $D_x : \Tex{Dom}(D_x) \to E_x$ is selfadjoint and regular, it suffices to verify that
\[
i r + x D x^* : x\big( \Tex{Dom}(D) \big) \to E_x
\]
has dense image whenever $r \in \rr$ satisfies $|r| > \| d(x^*) x \|_\infty$, see \cite[Lemma 9.7 and Lemma 9.8]{Lan:HMT}. Let such an $r \in \rr$ be given.

Clearly, $x^* x : E \to E$ satisfies the condition of Assumption \ref{a:self} and it therefore follows from Proposition \ref{p:reg} that $D x^* x : \Tex{Dom}(D x^* x) \to E$ is regular with core $\Tex{Dom}(D) \su E$. Combining this with Lemma \ref{l:inv} we may find a norm-dense submodule $\sE \su E$ such that
\[
(i r + D x^* x)^{-1}( \sE) = \Tex{Dom}(D) \, .
\]
Moreover, we have that
\[
(i r + x D x^*) x (i r + D x^* x)^{-1}(\xi) = x( \xi ) \q \Tex{for all } \xi \in \sE \, .
\]
Since $x(\sE) \su E_x$ is norm-dense and $x(ir + D x^* x)^{-1}(\sE) = x\big( \Tex{Dom}(D) \big)$, this proves the desired density result and hence that the localization $D_x : \Tex{Dom}(D_x) \to E_x$ is selfadjoint and regular.

Let $\mu \in \rr\sem\{0\}$. The identity in Equation \eqref{eq:reside} can now be verified on the image of $i \mu + x D x^* : x\big( \Tex{Dom}(D) \big) \to E_x$, but here it follows immediately since $(x D x^*)(\xi) = (D \De - d(x) x^*)(\xi)$ for all $\xi \in x\big( \Tex{Dom}(D) \big)$.
\end{proof}

\begin{rem}
The result of Proposition \ref{p:loca} can be generalized by replacing the bounded adjointable operator $x : E \to E$ by a sequence of bounded adjointable operators $x_n : E \to E$, $n \in \nn$, each of them satisfying the conditions of Assumption \ref{a:loca}. Suppose moreover that the sums
\[
\sum_{n = 1}^\infty x_n x_n^* \q \Tex{and} \q \sum_{n = 1}^\infty d(x_n) d(x_n)^*
\] 
are norm-convergent in $\mathbb{L}(E)$ (this can of course always be obtained by rescaling the operators $x_n : E \to E$, $n \in \nn$). 

In this context, we define the localization of $E$ with respect to the sequence $x = \{x_n\}$ as the closed submodule
\[
E_x := \Tex{cl}\big( \Tex{span}_{\cc}\{ x_n (\xi) \mid n \in \nn, \, \xi \in E \} \big) \su E \, .
\]
The localization $D_x$ of $D : \Tex{Dom}(D) \to E$ is defined as the closure of the symmetric unbounded operator
\[
\sum_{n = 1}^\infty x_n D x_n^* :  \Tex{Dom}(D) \cap E_x \to E_x  \, .
\]
As in Proposition \ref{p:loca}, we then obtain that $D_x : \Tex{Dom}(D_x) \to E_x$ is a selfadjoint and regular unbounded operator.
\end{rem}




\section{Localization of half-closed chains}\label{s:locahalf}
Let $A$ and $B$ be $\si$-unital $C^*$-algebras. Throughout this section $(\sA,E,D)$ will be a half-closed chain from $A$ to $B$. We denote by $\phi : A \to \mathbb{L}(E)$ the $*$-homomorphism that provides the left action of $A$ on $E$. Moreover, $x \in \sA$ will be a fixed element.

Notice that $\phi(x) : E \to E$ satisfies the condition of Assumption \ref{a:loca} with respect to the symmetric and regular unbounded operator $D : \Tex{Dom}(D) \to E$. Recall then that the localization of $E$ is the norm-closed submodule $E_x := \Tex{cl}\big( \Tex{Im}( \phi(x)) \big) \su E$ and that the localization $D_x$ of $D : \Tex{Dom}(D) \to E$ is the closure of the symmetric unbounded operator
\[
\phi(x) D \phi(x^*) : \Tex{Dom}(D) \cap E_x \to E_x \, .
\]
By Proposition \ref{p:loca}, the localization $D_x : \Tex{Dom}(D_x) \to E_x$ is selfadjoint and regular. We put
\[
\De := x x^* \in \sA \, .
\]

By definition, the {\it localization} of $A$ with respect to $x \in A$ is the hereditary $C^*$-subalgebra of $A$ defined by
\[
A_x := \Tex{cl}( x A x^*) \su A \, .
\]
The $*$-homomorphism $\phi : A \to \mathbb{L}(E)$ restricts to a $*$-homomorphism $\phi_x : A_x \to \mathbb{L}(E_x)$ and in this way $E_x$ becomes a $C^*$-correspondence from $A_x$ to $B$. We remark that $\De \in A_x$ and that $\phi_x(\De) : E_x \to E_x$ is a bounded positive and selfadjoint operator with norm-dense image. 

We define the $*$-subalgebra $\sA_x \su A_x$ as the intersection
\[
\sA_x := \sA \cap A_x \, .
\]
Remark that $\sA_x \su A_x$ is automatically norm-dense.

When the half-closed chain $(\sA,E,D)$ is even with $\zz/2\zz$-grading operator $\ga : E \to E$, then $E_x$ can be equipped with the $\zz/2\zz$-grading operator $\ga|_{E_x} : E_x \to E_x$ obtained by restriction of $\ga : E \to E$.

We are going to prove the following:

\begin{thm}\label{t:halfmodu}
Suppose that $(\sA,E,D)$ is a half-closed chain and that $x$ is an element in $\sA$. Then the triple $(E_x, D_x, \phi_x(\De) )$ is an unbounded modular cycle from $\sA_x$ to $B$ of the same parity as $(\sA,E,D)$ and with grading operator $\ga|_{E_x}: E_x \to E_x$ in the even case.
\end{thm}
\begin{proof}
Clearly the $C^*$-correspondence $E_x$ is countably generated (since $E$ is countably generated by assumption). Moreover, we have already established that the unbounded operator $D_x : \Tex{Dom}(D_x) \to E_x$ is selfadjoint and regular in Proposition \ref{p:loca} and that $\phi_x(\De) : E_x \to E_x$ is bounded positive and selfadjoint with norm-dense image. So it only remains to check conditions $(1)$, $(2)$, $(3)$ and $(4)$ of Definition \ref{d:modular}. The last condition $(4)$ follows immediately since $\De(\De + 1/n)^{-1} a \to a$ in $C^*$-norm for all $a \in A_x$. The remaining three conditions are proved in Proposition \ref{p:locacomp}, Proposition \ref{p:twist} and Proposition \ref{p:supremum} below.
\end{proof}

We will refer to the unbounded modular cycle $(E_x,D_x,\phi_x(\De))$ as the \emph{localization} of the half-closed chain $(E,\phi,D)$ with respect to $x \in \sA$.


We start by proving the compactness condition $(1)$ of Definition \ref{d:modular}. We put 
\[
\wit{D_x} := D \phi(\De) - d( x ) \phi(x^*) : \Tex{Dom}(D \phi(\De)) \to E
\]
and recall that $\wit{D_x}$ is a selfadjoint and regular unbounded operator by Lemma \ref{l:self}. We remark that $\wit{D_x}$ agrees with $D_x$ if and only if the image of $\phi(x) : E \to E$ is norm-dense. In fact, when the image of $\phi(x)$ is not norm-dense then these two unbounded operators do not even act on the same Hilbert $C^*$-module. 


\begin{lma}\label{l:locareso}
We have the resolvent identity
\[
\begin{split}
& \ma{cc}{0 & \phi(\De) \\ \phi(\De) & 0} \ma{cc}{(i + \wit{D_x})^{-1} & 0 \\ 0 & (i + \wit{D_x})^{-1} } 
- \ma{cc}{i & D^* \\ D & i}^{-1} \\
& \q =  \ma{cc}{i & D^* \\ D & i}^{-1} \ma{cc}{ d(x) \phi(x^*) - i & i \phi(\De) \\ i \phi(\De) & d(x) \phi(x^*) - i} (i + \wit{D_x})^{-1} \, .
\end{split}
\]
\end{lma}
\begin{proof}
It suffices to notice that the identities
\[
\begin{split}
& \ma{cc}{i & D^* \\ D & i} \ma{cc}{0 & \phi(\De) \\ \phi(\De) & 0} - \ma{cc}{i + \wit{D_x} & 0 \\ 0 & i + \wit{D_x}} \\
& \q = \ma{cc}{D \phi(\De) - i - \wit{D_x} & i \phi(\De) \\ i \phi(\De) & D \phi(\De) - i - \wit{D_x}} \\
& \q = \ma{cc}{d( x) \phi(x^*) - i & i \phi(\De) \\ i \phi(\De) & d( x) \phi(x^*) - i}
\end{split}
\]
hold on $\Tex{Dom}(D \phi(\De) ) \op \Tex{Dom}(D \phi(\De))$. Recall in this respect that $D \phi(\De) = D^* \phi(\De)$ by Proposition \ref{p:reg}.
\end{proof}

\begin{prop}\label{p:locacomp}
The bounded adjointable operator
\[
\phi_x(a) ( i + D_x)^{-1} : E_x \to E_x
\]
is compact for all $a \in A_x$.
\end{prop}
\begin{proof}
Notice that $\De \in A_x$ and that the left ideal $A_x \cd \Delta \su A_x$ is norm-dense. It thus suffices to show that $\phi_x(\De) \cd (i + D_x)^{-1} \in \mathbb{K}(E_x)$.

We apply the notation $\mathbb{K}(E,E_x) \su \mathbb{K}(E)$ for the closed right ideal generated by all compact operators on $E$ of the form $\ket{\xi}\bra{\eta}$ with $\xi \in E_x$ and $\eta \in E$. Similarly, we let $\mathbb{K}(E_x,E) \su \mathbb{K}(E)$ denote the closed left ideal generated by all compact operators of the form $\ket{\eta}\bra{\xi}$ for $\xi \in E_x$ and $\eta \in E$. We remark that $\mathbb{K}(E_x,E) = \mathbb{K}(E,E_x)^*$. 

Since $(E,\phi,D)$ is a half-closed chain we know that 
\[
\ma{cc}{\phi(\De) & 0 \\ 0 & \phi(\De) }\ma{cc}{i & D^* \\ D & i}^{-1} \in \mathbb{K}(E \op E)
\]
and it therefore follows from Lemma \ref{l:locareso} that
\[
\phi(\De)^2  (i + \wit{D_x})^{-1} \in \mathbb{K}(E,E_x) \, .
\]
Since $(\phi(\Delta)+1/n)^{-1}\phi(\Delta)^2  \to \phi(\Delta)$ as $n \to \infty$ this implies that also $\phi(\De) (i + \wit{D_x})^{-1} \in \mathbb{K}(E,E_x)$ and thus that $(-i + \wit{D_x})^{-1} \phi(\De) \in \mathbb{K}(E_x,E)$. We may thus conclude that $\phi(\De) (1 + \wit{D_x}^2)^{-1} \phi(\De) \in \mathbb{K}(E,E_x) \cd \mathbb{K}(E_x,E)$ restricts to a compact operator on the Hilbert $C^*$-module $E_x \subseteq E$. But this proves the present proposition since we have from Proposition \ref{p:loca} that
\[
\phi_x(\De) (1 + D_x^2)^{-1} \phi_x(\De) = \big( \phi(\De) (1 + \wit{D_x}^2)^{-1} \phi(\De) \big)|_{E_x} \, .\qedhere
\]
\end{proof}

We continue by proving the twisted commutator condition $(2)$ of Definition \ref{d:modular}.

\begin{prop}\label{p:twist}
Let $a \in \sA_x$, $\la \in \cc$. Then $(\phi_x(a) + \la)\phi_x(\De) : E_x \to E_x$ has $\Tex{Dom}(D_x) \su E_x$ as an invariant submodule and
\[
D_x (\phi_x(a) + \la)\phi_x(\De) - \phi_x(\De) (\phi_x(a) + \la) D_x : \Tex{Dom}(D_x) \to E_x
\]
extends to a bounded adjointable operator $d_\De(a,\la) : E_x \to E_x$. In fact we have that
\[
\begin{split}
d_\De(a,\la) & = \big( \phi(x) d( x^* (a + \la) x ) \phi(x^*) \big)|_{E_x}\, .
\end{split}
\]
\end{prop}
\begin{proof}
Let $\xi \in \Tex{Dom}(D) \cap E_x$. We then have that
\[
( \phi_x(a) + \la ) \phi_x(\De)(\xi) \in \Tex{Dom}(D) \cap E_x
\]
and that
\[
\begin{split}
& D_x (\phi_x(a) + \la)\phi_x(\De)(\xi) - \phi_x(\De) (\phi_x(a) + \la) D_x(\xi) \\
& \q =
\phi(x) D \phi(x^*) (\phi(a) + \la)\phi(x x^*)(\xi)  \\
& \qqq - \phi(x x^*) (\phi(a) + \la) \phi(x) D \phi(x^*) (\xi) \\
& \q = 
\phi(x) d( x^* (a + \la) x ) \phi(x^*)(\xi) \, . 
\end{split}
\]
Since $\Tex{Dom}(D) \cap E_x$ is a core for the localization $D_x : \Tex{Dom}(D_x) \to E_x$, this proves the proposition.
\end{proof}

We finally prove the supremum condition $(3)$ of Definition \ref{d:modular}.

\begin{prop}\label{p:supremum}
Let $a \in \sA_x$, $\la \in \cc$. Then we have that
\[
\sup_{\ep > 0} \| (\phi_x(\De)^{1/2} + \ep)^{-1} d_\De(a,\la) (\phi_x(\De)^{1/2} + \ep)^{-1} \|_\infty < \infty \, .
\]
\end{prop}
\begin{proof}
This follows immediately from Proposition \ref{p:twist}. Indeed, the operator norm of
\[
(\phi_x(\De)^{1/2} + \ep)^{-1} \phi(x) : E \to E_x
\]
is bounded by $1$ for all $\ep > 0$.
\end{proof}

\begin{rem}
One may equip $\sA_x$ with the operator space norm $\| \cd \|_1 : M_n(\sA_x) \to [0,\infty)$, $n \in \nn$, defined by
\[
\| a \|_1 := \sup\{ \| a \| , \| d(a) \|_\infty \} \q \Tex{ for all } a \in M_n(\sA_x) \, ,
\]
where the norms inside the supremum are the $C^*$-norm on $M_n(A)$ and the operator-norm on $\mathbb{L}(E^{\op n})$, respectively. Clearly, the inclusion $\sA_x \to A_x$ is then completely bounded. It is moreover possible to find a constant $C > 0$ such that
\[
\sup_{\ep > 0} \| (\phi_x(\De)^{1/2} + \ep)^{-1} d_\De(a,0) (\phi_x(\De)^{1/2} + \ep)^{-1} \|_\infty \leq C \cd \| a \|_1 \, ,
\]
for all $a \in M_n(\sA_x)$. Cf. Remark \ref{r:cbb}.
\end{rem}




\section{Localization as an unbounded Kasparov product}\label{s:unbkas}
In this section we continue under the conditions spelled out in the beginning of Section \ref{s:locahalf}. We thus have a half-closed chain $(\sA,E,D)$ and an element $x \in \sA$. 

The element $x \in \sA$ provides us with a closed right ideal $I_x \su A$ defined as the norm-closure:
\[
I_x := \Tex{cl}( x A) \, .
\]
In particular, we may consider $I_x$ as a countably generated Hilbert $C^*$-module over $A$. The hereditary $C^*$-subalgebra $A_x = \Tex{cl}(xAx^*) \subseteq A$ can be identified with the compact operators on $I_x$ via the $*$-homomorphism $\psi : A_x \to \mathbb{L}(I_x)$ induced by the multiplication in $A$. We thus obtain an even Kasparov module $(I_x,0)$ from $A_x$ to $A$ with corresponding class $[I_x,0] \in KK_0(A_x,A)$ in KK-theory.

Moreover, by Theorem \ref{t:kasphalf}, our half-closed chain $(\sA,E,D)$ (of parity $p \in \{0,1\}$) yields a Kasparov module $(E,F_D)$ from $A$ to $B$ with corresponding class $[E,F_D] \in KK_p(A,B)$.

Finally, the unbounded modular cycle $(\sA \cap A_x, E_x, \phi_x(\De))$ constructed in Section \ref{s:locahalf} yields a Kasparov module $(E_x,F_{D_x})$ from $A_x$ to $B$ with corresponding class $[E_x,F_{D_x}] \in KK_p(A_x,B)$, see Theorem \ref{t:modkas}.

In this section we will prove the following theorem:

\begin{thm}
\label{thm:loca-product}
Suppose that $(\sA,E,D)$ is a half-closed chain, that $x \in \sA$ and that $A_x$ is separable. Then we have the identity
\[
[ E_x,F_{D_x}] = [ I_x,0] \hot_A [E, F_D]
\]
in $KK_p(A_x,B)$, where $\hot_A : KK_0(A_x,A) \ti KK_p(A,B) \to KK_p(A_x,B)$ denotes the Kasparov product.
\end{thm}
\begin{proof}
The $C^*$-correspondence $E_x$ from $A_x$ to $A$ is unitarily isomorphic to the interior tensor product of $C^*$-correspondences $I_x \hot_\phi E$ (via the unitary isomorphism $xa \hot \xi \mapsto \phi(xa)(\xi)$). For each $a \in A$, we define the bounded adjointable operator $T_{xa} : E \to E_x$ by $\xi \mapsto \phi(xa)(\xi)$. By \cite[Theorem A.3]{CoSk:LIF} it suffices to prove the connection condition, thus that
\begin{equation}\label{eq:conncondI}
\begin{split}
& F_{D_x} T_{xa} -  T_{xa} F_D \, , \\
& F_{D_x} T_{xa}  - T_{xa} F_{D^*} \in \mathbb{K}(E,E_x)
\end{split}
\end{equation}
for all $a \in A$. Indeed, the positivity condition of \cite[Theorem A.3]{CoSk:LIF} is obviously satisfied since the bounded adjointable operator in the Kasparov module $(I_x,0)$ from $A_x$ to $A$ is trivial. See also Section \ref{s:kuce} for more details.

However, since $T_{xa} = T_x \phi(a) : E \to E_x$ and $\phi(a) (F_D - F_{D^*}) \in \mathbb{K}(E)$ it suffices to prove the first of these inclusions. This proof will occupy the remainder of this section, see Proposition \ref{p:connectionI}.
\end{proof}

\begin{rem}
In the case where $xA \su A$ is norm-dense and $A$ is separable, we have that $(I_x,0) = (A, 0)$ and it therefore follows from the above theorem that the two Kasparov modules $(E_x,F_{D_x})$ and $(E,F_D)$ represents the same class in $KK_p(A,B)$. 
\end{rem}




\subsection{The modular transform}\label{ss:modular}
We continue working under the general assumptions stated in the beginning of Section \ref{s:locahalf}. We recall that $\De := x x^*$. We will in the following suppress the $*$-homomorphism $\phi_x : A_x \to \mathbb{L}(E_x)$.

For each $\la \geq 0$, we introduce the notation
\[
\begin{split}
& R_x(\la \De^2) := (1 + \la \De^2 + D_x^2)^{-1} \in \mathbb{L}(E_x) \\
& R_x(\la) := (1 + \la + D_x^2)^{-1} \in \mathbb{L}(E_x) \, .
\end{split}
\]
In general, we are not able to estimate the norm of $R_x(\la \De^2)$ from above by $(1 + \la)^{-1}$ since $\De : E_x \to E_x$ may have zero in the spectrum. Instead, we recall the following basic estimate from \cite[Section 11]{Kaa:UKM}:
\begin{equation}\label{eq:basic}
\| \De R_x(\la \De^2) \De \|_\infty \leq \frac{2}{(1 + \la)} \q \forall \la \geq 0 \, .
\end{equation}

The next definition is from \cite[Section 8]{Kaa:UKM}:

\begin{defn}
The \emph{modular transform} of the unbounded modular cycle $(E_x,D_x, \De)$ is the unbounded operator
\[
G_{(D_x,\De)} : \De( \Tex{Dom}(D_x) ) \to E_x
\]
defined by
\begin{equation}\label{eq:modutrans}
G_{(D_x,\De)} : \eta \mapsto 
\frac{1}{\pi} \int_0^\infty \la^{-1/2} \De (1 + \la \De^2 + D_x^2  )^{-1} D_x(\eta) \, d\la \, .
\end{equation}
\end{defn}

We remark that $G_{(D_x,\De)} : \De( \Tex{Dom}(D_x) ) \to E_x$ is well-defined. Indeed, for $\eta = \De(\xi)$ with $\xi \in \Tex{Dom}(D_x)$ we have from Proposition \ref{p:twist} that
\[
\begin{split}
& \De R_x( \la \De^2) D_x(\eta) \\
& \q = \De R_x(\la \De^2) \De D_x(\xi)
+ \De R_x(\la \De^2) x d(x^* x) x^* (\xi) \, .
\end{split}
\]
Using the estimate from Equation \eqref{eq:basic}, we may thus find a constant $C > 0$ such that
\[
\| \De (1 + \la \De^2 + D_x^2  )^{-1} D_x(\eta) \| \leq C \cd (1 + \la)^{-3/4} \q \forall \la \geq 0 \, ,
\]
implying that the integral in Equation \eqref{eq:modutrans} converges absolutely in the norm on $E_x$.

The following result is a consequence of \cite[Theorem 8.1]{Kaa:UKM}:

\begin{thm}\label{t:modumodu}
The difference
\[
F_{D_x} \De^6 - G_{(D_x,\De)} \De^6 : \Tex{Dom}(D_x) \to E_x
\]
extends to a compact operator on $E_x$.
\end{thm}

Notice that the above result implies that the unbounded operator
\[
G_{(D_x,\De)} \De^6 : \Tex{Dom}(D_x) \to E_x
\]
extends to a bounded adjointable operator on $E_x$.




\subsection{The connection condition}
We will continue working under the assumptions of Section \ref{s:locahalf}.

We recall from Lemma \ref{l:self} that
\[
\wit{D_x} := D \phi(\De) - d(x) \phi(x^*) : \Tex{Dom}(D \phi(\De)) \to E
\]
is a selfadjoint and regular unbounded operator and we put
\[
\begin{split}
& \wit{R_x}( \la \phi(\De^2)) := (1 + \la \phi(\De^2) + (\wit{D_x})^2)^{-1} \in \mathbb{L}(E) \\
& R(\la) := (1 + \la + D^* D)^{-1} \in \mathbb{L}(E) \, ,
\end{split}
\]
for all $\la \geq 0$.

\begin{lma}\label{l:resolvent}
For each $\la \geq 0$, we have the identity
\[
\begin{split}
& R(\la) - \wit{R_x}( \la \phi(\De^2)) \phi(\De^2) \\
& \q = \wit{R_x}(\la \phi(\De^2)) \big( 1 - \phi(\De^2) + \phi(x) d(x^* x x^*) D\big) R(\la) \\
& \qqq + \big( \wit{D_x} \wit{R_x}(\la \phi(\De^2) ) \big)^* \phi(x) d(x^*) R(\la)
\end{split}
\]
of bounded adjointable operators on $E$.
\end{lma}
\begin{proof}
We have the identities
\[
\begin{split}
& 1 - \wit{R_x}(\la \phi(\De^2)) \phi(\De^2) (1 + \la + D^* D) \\
&\q  = 1 - \wit{R_x}(\la \phi(\De^2)) \big( 1 + \la \phi(\De^2) + \phi(x) D \phi(x^* x x^*) D \big) \\
& \qq + \wit{R_x}( \la \phi(\De^2)) ( 1 - \phi(\De^2) + \phi(x) d(x^* x x^*) D) \\
& \q = \big( \wit{D_x} \wit{R_x}(\la \phi(\De^2)) \big)^* \phi(x) d(x^*) \\
& \qq + \wit{R_x}( \la \phi(\De^2)) ( 1 - \phi(\De^2) + \phi(x) d(x^* x x^*) D)
\end{split}
\]
on $\Tex{Dom}(D^* D)$. But this proves the lemma after multiplying with $R(\la) = (1 + \la + D^* D)^{-1}$ from the right.
\end{proof}

For each $y \in I_x = \Tex{cl}(xA)$, we recall that $T_y : E \to E_x$ denotes the bounded adjointable operator $T_y : \xi \mapsto \phi(y)(\xi)$. Notice then that it follows from Proposition \ref{p:loca} that 
\[
T_\De \wit{R_x}(\la \phi(\De^2)) \phi(\De) = \De R_x(\la \De^2) T_\De : E \to E_x \, .
\]

\begin{lma}\label{l:resest}
The difference
\[
T_\De R(\la) D \phi(\De) - \De R_x(\la \De^2) D_x T_{\De^2} : \Tex{Dom}(D \phi(\De)) \to E_x
\]
extends to a compact operator $M_\la : E \to E_x$ for all $\la \geq 0$. Moreover, there exists a constant $C > 0$ such that
\[
\| M_\la \|_\infty \leq C \cd (1 + \la)^{-3/4} \q \forall \la \geq 0 \, .
\]
\end{lma}
\begin{proof}
Since $(\sA,E,D)$ is a half-closed chain and $(E_x, D_x, \De)$ is an unbounded modular cycle we obtain that the difference
\[
T_\De R(\la) D \phi(\De) - \De R_x(\la \De^2) D_x T_{\De^2} : \Tex{Dom}(D \phi(\De)) \to E_x
\]
extends to a compact operator $M_\la : E \to E_x$ for all $\la \geq 0$. Indeed, this is already true for each of the terms viewed separately. So we only need to prove the norm-estimate. To this end, we let $\xi \in \Tex{Dom}(D \phi(\De))$ and compute that
\[
\begin{split}
& ( T_\De R(\la) D \phi(\De) - \De R_x(\la \De^2) D_x T_{\De^2})(\xi) \\
& \q = T_\De R(\la) D \phi(\De)(\xi) - T_\De \wit{R_x}(\la \phi(\De^2)) \phi(x) D \phi(x^* \De^2)(\xi) \\
& \q = T_\De R(\la) D \phi(\De)(\xi) - T_\De \wit{R_x}(\la \phi(\De^2))  \phi(\De^2) D \phi(\De)(\xi) \\
& \qqq  - T_\De \wit{R_x}(\la \phi(\De^2)) \phi(x) d(x^* x x^*) \phi(\De)(\xi) \, .
\end{split}
\]
Since $\| T_\De \wit{R_x}(\la \phi(\De^2)) \phi(x) \|_\infty \leq 2^{3/4} \cd (1 + \la)^{-3/4}$ by the estimate in Equation \eqref{eq:basic} we may focus on the difference
\[
T_\De R(\la) D \phi(\De)(\xi) - T_\De \wit{R_x}(\la \phi(\De^2)) \phi(\De^2) D \phi(\De)(\xi) \, .
\]
However, using Lemma \ref{l:resolvent} we get that
\[
\begin{split}
& T_\De R(\la) D \phi(\De)(\xi) - T_\De \wit{R_x}(\la \phi(\De^2)) \phi(\De^2) D \phi(\De)(\xi) \\
& \q = T_\De \wit{R_x}(\la \phi(\De^2)) \big( 1 - \phi(\De^2) + \phi(x) d(x^* x x^*) D\big)  R(\la) D \phi(\De)(\xi) \\
& \qqq + T_\De \big( \wit{D_x} \wit{R_x}(\la \phi(\De^2) ) \big)^* \phi(x) d(x^*) R(\la) D \phi(\De)(\xi) \, .
\end{split}
\]
The result of the lemma then follows from the basic estimate $\| D R(\la) \|_\infty \leq (1 + \la)^{-1/2}$ and the estimate in Equation \eqref{eq:basic} a few times.
\end{proof}

\begin{prop}\label{p:stramodu}
The difference
\[
T_{\De^2} F_D - G_{(D_x,\De)} T_{\De^2} : \Tex{Dom}(D) \to E_x
\]
extends to a compact operator from $E$ to $E_x$.
\end{prop}
\begin{proof}
Since $\phi(\De) F_D - F_{D^*} \phi(\De) : E \to E$ is compact, we only need to show that
\[
T_\De F_{D^*} \phi(\De) - G_{(D_x,\De)} T_{\De^2} : \Tex{Dom}(D) \to E_x
\]
extends to a compact operator from $E$ to $E_x$. Now, recall that
\[
T_\De F_{D^*} \phi(\De)(\xi)
= \frac{1}{\pi} \int_0^\infty \la^{-1/2} T_\De (1 + \la + D^* D)^{-1} D \phi(\De)(\xi) \, d\la
\]
for all $\xi \in \Tex{Dom}(D)$. The result of the proposition now follows by Lemma \ref{l:resest} since
\[
\begin{split}
& T_\De D^* (1 + DD^*)^{-1/2} \phi(\De)(\xi) - 
G_{(D_x,\De)} T_{\De^2}(\xi) \\
& \q = \frac{1}{\pi} \int_0^\infty \la^{-1/2} \big( T_\De (1 + \la + D^* D)^{-1} D \phi(\De) \\ 
& \qqq \qqq - \De R_x(\la \De^2) D_x T_{\De^2} \big)(\xi) \, d\la \\
& \q = \frac{1}{\pi} \int_0^\infty \la^{-1/2} M_\la(\xi) \, d\la \, . \qedhere
\end{split}
\]
\end{proof}

Remark that it follows from the above proposition that the unbounded operator
\[
G_{(D_x,\De)} T_{\De^2} : \Tex{Dom}(D) \to E_x
\]
extends to a bounded adjointable operator on $E_x$.

\begin{prop}\label{p:connectionI}
The difference
\[
F_{D_x} T_{xa} -  T_{xa} F_D : E \to E_x
\]
is a compact operator for all $a \in A$.
\end{prop}
\begin{proof}
Since $[\phi(b), F_D ] \in \mathbb{K}(E)$ for all $b \in A$ and since $\De^7(1/n + \De^7)^{-1} x \to x$ in the norm on $A$, it suffices to show that
\[
F_{D_x} T_{\De^7} - T_{\De^7} F_D : E \to E_x
\]
is a compact operator. But now Proposition \ref{p:stramodu} and Theorem \ref{t:modumodu} imply that the following identities hold modulo $\mathbb{K}(E,E_x)$:
\[
\begin{split}
F_{D_x} T_{\De^7} - T_{\De^7} F_D 
& \sim F_{D_x} T_{\De^7}  - T_{\De^2} F_D \phi(\De^5) \\
& \sim F_{D_x} T_{\De^7} - \Tex{cl}( G_{(D_x,\De)} T_{\De^2} ) \phi(\De^5) \\
& = F_{D_x} \De^6 T_{\De} - \Tex{cl}( G_{(D_x,\De)} \De^6 ) T_{\De} \sim 0 \, . \qedhere
\end{split}
\]
%
\end{proof}




\section{Kucerovsky's theorem}\label{s:kuce}
Let us fix three $C^*$-algebras $A,B$ and $C$ with $A$ separable and $B$ and $C$ both $\si$-unital. Throughout this section we will assume that $(\sA,E_1,D_1)$, $(\sB,E_2,D_2)$ and $(\sA,E,D)$ are even half-closed chains from $A$ to $B$, from $B$ to $C$ and from $A$ to $C$, respectively. We denote the associated $*$-homomorphisms by $\phi_1 : A \to \mathbb{L}(E_1)$, $\phi_2 : B \to \mathbb{L}(E_2)$ and $\phi : A \to \mathbb{L}(E)$ and the $\zz/2\zz$-grading operators by $\ga_1 : E_1 \to E_1$, $\ga_2 : E_2 \to E_2$ and $\ga : E \to E$, respectively. We will moreover assume that $E := E_1 \hot_{\phi_2} E_2$ agrees with the interior tensor product of the $C^*$-correspondences $E_1$ and $E_2$. In particular, we assume that $\phi(a) = \phi_1(a) \hot 1$ for all $a \in A$ and that $\ga = \ga_1 \hot \ga_2$. 

We will denote the bounded transforms of our half-closed chains by $(E_1, F_{D_1})$, $(E_2,F_{D_2})$ and $(E,F_D)$ and the corresponding classes in KK-theory by $[E_1,F_{D_1}] \in KK_0(A,B)$, $[E_2,F_{D_2}] \in KK_0(B,C)$ and $[E,F_D] \in KK_0(A,C)$. We may then form the interior Kasparov product
\[
[E_1, F_{D_1}] \hot_B [E_2,F_{D_2}] \in KK_0(A,C)
\]
and it becomes a highly relevant question to find an explicit formula for this class in $KK_0(A,C)$. 

In this section we shall find conditions on the half-closed chains $(\sA,E_1,D_1)$, $(\sB,E_2,D_2)$ and $(\sA,E,D)$ entailing that the identity
\[
[E,F_D] = [E_1,F_{D_1}] \hot_B [E_2,F_{D_2}]
\]
holds in $KK_0(A,C)$. This kind of theorem was first proved by Kucerovsky in \cite{Kuc:KUM} under the stronger assumption that the half-closed chains $(\sA,E_1,D_1)$, $(\sB,E_2,D_2)$ and $(\sA,E,D)$ were in fact unbounded Kasparov modules. Thus under the strong assumption that all the involved symmetric and regular unbounded operators were in fact \emph{selfadjoint}. As in the case of Kucerovsky's theorem we rely on the work of Connes and Skandalis for computing the interior Kasparov product, see \cite{CoSk:LIF}.

We recall from \cite[Theorem A.3]{CoSk:LIF} that an even Kasparov module $(E,F)$ from $A$ to $C$ is the \emph{Kasparov product} of the even Kasparov modules $(E_1,F_1)$ and $(E_2,F_2)$ from $A$ to $B$ and from $B$ to $C$, respectively, when the following holds:
\begin{itemize}
\item $E = E_1 \hot_{\phi_2} E_2$, $\phi = \phi_1 \hot 1$.
\item For every homogeneous $\xi \in E_1$ we have that
\begin{equation}\label{eq:conncond}
F T_\xi - (-1)^{\pa \xi} T_\xi F_2 \, , \, \, F^* T_\xi - (-1)^{\pa \xi} T_\xi F_2^* \in \mathbb{K}(E_2,E) \, ,
\end{equation}
where $T_\xi : E_2 \to E$ is defined by $T_\xi(y) := \xi \hot \eta$ for all $\eta \in E_2$ and where $\pa \xi \in \{0,1\}$ denotes the degree of $\xi \in E_1$.
\item There exists a $\nu < 2$ such that
\begin{equation}\label{eq:posicond}
\big(  (F_1 \hot 1)^* \cd F^* + F \cd (F_1 \hot 1) \big) \cd \phi(a^*a) + \nu \cd \phi(a^* a)
\end{equation}
is positive in the Calkin algebra $\mathbb{L}(E)/\mathbb{K}(E)$ for all $a \in A$.
\end{itemize}

The condition in Equation \eqref{eq:conncond} is often referred to as the \emph{connection condition} and the condition in Equation \eqref{eq:posicond} is referred to as the \emph{positivity condition}.

Before we state our conditions on half-closed chains we recall that the odd symmetric and regular unbounded operator $D_1 : \Tex{Dom}(D_1) \to E_1$ can be promoted to an odd symmetric and regular unbounded operator $D_1 \hot 1 : \Tex{Dom}(D_1 \hot 1) \to E_1 \hot_{\phi_2} E_2$ with resolvent $(1 + D_1^* D_1)^{-1} \hot 1 \in \mathbb{L}(E_1 \hot_{\phi_2} E_2)$.
%

We now introduce the analogues for the above connection and positivity condition for half-closed chains. They will be shown in Theorem \ref{t:kuce} below to indeed correspond to the above two conditions for Kasparov modules.

\begin{defn}
\label{defn:conn-cond}
Given three even half-closed chains $(\sA,E_1,D_1)$, \newline $(\sB,E_2,D_2)$ and $(\sA,E_1 \hot_{\phi_2} E_2,D)$ as above, the {\em connection condition} demands that there exist a dense $\sB$-submodule $\sE_1 \su E_1$ and cores $\sE_2$ and $\sE$ for $D_2 : \Tex{Dom}(D_2) \to E_2$ and $D : \Tex{Dom}(D) \to E$, respectively, such that
\begin{itemize}
\item[(a)] For each $\xi \in \sE_1$:
\[
T_\xi( \sE_2 ) \su \Tex{Dom}(D) \, \, , \q T_\xi^*(\sE) \su \Tex{Dom}(D_2) \, \, , \q \ga_1(\xi) \in \sE_1 \, .
\]
\item[(b)] For each homogeneous $\xi \in \sE_1$, the graded commutator
\[
D T_\xi - (-1)^{\pa \xi} T_\xi D_2 : \sE_2 \to E
\]
extends to a bounded operator $L_\xi : E_2 \to E$.
\end{itemize}
\end{defn}


\begin{defn}
\label{defn:loca-subset}
Given three even half-closed chains $(\sA,E_1,D_1)$, \newline $(\sB,E_2,D_2)$ and $(\sA,E_1 \hot_{\phi_2} E_2,D)$ as above, a {\em localizing subset} is a \emph{countable} subset $\La \su \sA$ with $\La = \La^*$ such that
\begin{itemize}
\item[(a)] The subspace 
\[
\La \cd A := \Tex{span}_{\mathbb{C}}\{ x \cd a \mid x \in \La \, , \, \, a \in A \} \subseteq A
\]
is norm-dense. 
\item[(b)] The commutator
\[
[D_1 \hot 1, \phi(x) ] : \Tex{Dom}(D_1 \hot 1) \to E
\]
is  \emph{trivial} for all $x \in \La$.
\item[(c)] We have the domain inclusion
\[
\Tex{Dom}(D) \cap \Tex{Im}(\phi(x^* x)) \su \Tex{Dom}(D_1 \hot 1) \, ,
\]
for all $x \in \La$.
\end{itemize}
\end{defn}
%

\begin{defn}
\label{defn:loca-pos-cond}
Given three even half-closed chains $(\sA,E_1,D_1)$, \newline $(\sB,E_2,D_2)$ and $(\sA,E_1 \hot_{\phi_2} E_2,D)$ and a localizing subset $\Lambda \su \sA$, the {\em local positivity condition} requires that for each $x \in \La$, there exists a constant $\ka_x > 0$ such that 
\[
\begin{split}
& \binn{ (D_1 \hot 1) \phi(x^*) \xi, D \phi(x^*) \xi } + \inn{ D \phi(x^*) \xi, (D_1 \hot 1) \phi(x^*) \xi} \\
& \q \geq - \ka_x \cd \inn{\xi, \xi} \, ,
\end{split}
\]
for all $\xi \in \Tex{Im}(\phi( x)) \cap \Tex{Dom}( D \phi(x^*) )$.
\end{defn}

Note that the local positivity condition makes sense because of (d) in Definition \ref{defn:loca-subset}. Indeed, for each $\xi \in \Tex{Im}(\phi( x)) \cap \Tex{Dom}( D \phi(x^*) )$ we have that
\[
\phi(x^*) \xi \in \Tex{Im}(\phi(x^* x)) \cap \Tex{Dom}(D) \su \Tex{Dom}(D_1 \hot 1) \, .
\]

\begin{rem}
Suppose that $\sA \subseteq A$ is unital and that $\phi_1(A) \cd E_1 \su E_1$ is norm-dense. Then the half-closed chains $(\sA,E,D)$ and $(\sA,E_1,D_1)$ are in fact unbounded Kasparov modules (thus $D = D^*$ and $D_1 = D_1^*$). The choice $\La := \{1\} \subseteq \sA$ automatically satisfies the conditions (a) and (b) for a localizing subset in Definition \ref{defn:loca-subset} and the last condition (c) amounts to the requirement
\[
\Tex{Dom}(D) \subseteq \Tex{Dom}(D_1 \hot 1) \, .
\]
Moreover, in this case, the local positivity condition in Definition \ref{defn:loca-pos-cond} means that there exists a constant $\ka > 0$ such that
\[
\binn{ (D_1 \hot 1) \xi, D \xi} + \binn{D \xi, (D_1 \hot 1) \xi} \geq - \ka \cd \inn{\xi,\xi} \, ,
\]
for all $\xi \in \Tex{Dom}(D)$. Finally, the connection condition in Definition \ref{defn:conn-cond} can be seen to be equivalent to the connection condition applied by Kucerovsky in \cite{Kuc:KUM}. In this setting, we therefore recover the assumptions applied by Kucerovsky in \cite[Theorem 13]{Kuc:KUM} (except that the domain condition in \cite[Theorem 13]{Kuc:KUM} is marginally more flexible). The corresponding special case of Theorem \ref{t:kuce} here below, is therefore in itself an improvement to \cite[Theorem 13]{Kuc:KUM} because of the extra flexibility in the choice of localizing subset $\La \su \sA$ (if one is willing to disregard the minor domain issue mentioned earlier in this remark).
\end{rem}

We record the following convenient lemma, which can be proved by standard techniques:

\begin{lma}\label{l:fulldomain}
Suppose that the connection condition of Definition \ref{defn:conn-cond} holds. Then the connection condition holds for $\sE_2 := \Tex{Dom}(D_2)$ and $\sE := \Tex{Dom}(D)$. Moreover, $L_\xi : E_2 \to E$ is adjointable with
\[
(L_\xi)^*(\eta) = (T_\xi^* D - (-1)^{\pa \xi} D_2 T_\xi^*)(\eta) \q \forall \eta \in \Tex{Dom}(D)
\]
whenever $\xi \in \sE_1$ is homogeneous.
\end{lma}

The next lemma provides a convenient sufficient condition for verifying the inequality in Definition \ref{defn:loca-pos-cond}:

\begin{lma}
Let $x \in A$ and suppose that $\Tex{Im}(\phi(x^* x)) \cap \Tex{Dom}(D) \su \Tex{Dom}(D_1 \hot 1)$ and that there exists a constant $\ka_x > 0$ such that
\[
\inn{ (D_1 \hot 1) \eta, D \eta} + \inn{D \eta, (D_1 \hot 1) \eta} \geq - \ka_x \inn{\eta,\eta} \, ,
\]
for all $\eta \in \Tex{Im}(\phi(x^* x)) \cap \Tex{Dom}(D)$. Then we have that
\[
\begin{split}
& \binn{ (D_1 \hot 1) \phi(x^*) \xi, D \phi(x^*) \xi} + \binn{D \phi(x^*) \xi, (D_1 \hot 1) \phi(x^*) \xi} \\
& \q \geq - \| \phi(x) \|^2 \ka_x \inn{\xi,\xi} \, ,
\end{split}
\]
for all $\xi \in \Tex{Im}(\phi( x)) \cap \Tex{Dom}( D \phi(x^*) )$.
\end{lma}
\begin{proof}
This follows immediately since 
\[
- \ka_x \inn{ \phi(x^*) \xi, \phi(x^*) \xi} \geq - \| \phi(x) \|^2 \ka_x \inn{\xi,\xi} \q  \forall \xi \in E \, . \qedhere
\]
\end{proof}

The next lemma is straightforward to prove by rescaling the elements in $\La$ by elements in $(0,\infty)$. It will nonetheless play a very important role:

\begin{lma}\label{l:rescaling}
Suppose that the local positivity condition of Definition \ref{defn:loca-pos-cond} holds with localizing subset $\La \subseteq \sA$. Then we may rescale the elements in $\La$ and obtain a localizing subset $\La' \su \sA$ such that the local positivity condition of Definition \ref{defn:loca-pos-cond} holds with the additional requirement that
\[
\ka_x = 1/4 \q \Tex{and} \q \| d(x^*) \phi(x) \|_\infty < 1 \q \forall x \in \La' \, .
\]
\end{lma}

\begin{thm}\label{t:kuce}
Suppose that the three even half-closed chains $(\sA,E_1,D_1)$, $(\sB,E_2,D_2)$ and $(\sA,E_1 \hot_{\phi_2} E_2,D)$ satisfy the connection condition and the local positivity condition. Then $(E,F_D)$ is the Kasparov product of $(E_1,F_{D_1})$ and $(E_2,F_{D_2})$. In particular we have the identity
\[
[E,F_D] = [E_1,F_{D_1}] \hot_B [E_2,F_{D_2}]
\]
in the KK-group $KK_0(A,C)$.
\end{thm}
\begin{proof}
Without loss of generality we may assume that $\ka_x = 1/4$ and that $\| d(x^*) \phi(x) \|_\infty < 1$ for all $x \in \La$.

We need to prove the connection condition in Equation \eqref{eq:conncond} and the positivity condition in Equation \eqref{eq:posicond} for the even Kasparov modules
$(E,F_D)$, $(E_1,F_{D_1})$ and $(E_2,F_{D_2})$. 

But these two conditions are proved in Proposition \ref{p:connection} and Proposition \ref{p:positivity} below, respectively. The positivity condition will be satisfied with $\nu = 1 = 4 \cd \ka_x$.
\end{proof}




\subsection{The connection condition}
We continue working in the setting explained in the beginning of Section \ref{s:kuce}.

Before proving our first proposition on the connection condition in Equation \eqref{eq:conncond}, it will be convenient to introduce some extra notation. For $\la \in [0,\infty)$, define the bounded adjointable operators
\[
\begin{split}
R(\la) & := (1 + \la + D^* D)^{-1} \, \, , \,  \ov{R}(\la) := (1 + \la + D D^*)^{-1} : E \to E \\
R_2(\la) & := (1 + \la + D_2^* D_2)^{-1} \, \, , \, \ov{R_2}(\la) := (1 + \la + D_2 D_2^*)^{-1} : E_2 \to E_2 \, .
\end{split}
\]

\begin{prop}\label{p:connection}
Suppose that the connection condition of Definition \ref{defn:conn-cond} holds. Then we have that
\[
F_D T_\xi - (-1)^{\pa \xi} T_\xi F_{D_2} \, , \, \, F_{D}^* T_\xi - (-1)^{\pa \xi} T_\xi F_{D_2}^* \in \mathbb{K}(E_2,E) \, ,
\]
for all homogeneous $\xi \in E_1$. 
\end{prop}
\begin{proof}
Without loss of generality we may assume that $\xi = \eta \cd b_1 b_2$ with $\eta \in \sE_1$ homogeneous and $b_1,b_2 \in \sB$. Using Lemma \ref{l:fulldomain} we compute as follows, for each $\la \in [0,\infty)$:
\[
\begin{split}
& R(\la) T_{\eta \cd b_1} - T_{\eta \cd b_1} R_2(\la)
= R(\la) T_{\eta \cd b_1} D_2^* D_2 R_2(\la)
- D^* D R(\la) T_{\eta \cd b_1} R_2(\la) \\
& \q = - R(\la) T_\eta \cd d_2(b_1) \cd D_2 R_2(\la)
- (-1)^{\pa \eta} R(\la) L_\eta \cd \phi_2( b_1) \cd D_2 R_2(\la) \\
& \qqq + (-1)^{\pa \eta} R(\la) D T_{\eta \cd b_1} \cd D_2 R_2(\la)
- D^* D R(\la) T_{\eta \cd b_1} R_2(\la) \\
& \q = - R(\la) \big( T_\eta \cd d_2(b_1) + (-1)^{\pa \eta} L_\eta \cd \phi_2(b_1) \big) \cd D_2 R_2(\la) \\
& \qqq - D^* \ov{R}(\la) L_{\eta \cd b_1} R_2(\la) \, ,
\end{split}
\]
where $d_2(b_1) : E_2 \to E_2$ is the bounded extension of the commutator $D_2 \phi_2(b_1) - \phi_2(b_1) D_2^* : \Tex{Dom}(D_2^*) \to E$. In particular, we may find a constant $C > 0$ such that
\begin{equation}\label{eq:estimate}
\big\| D R(\la) T_{\eta \cd b_1} - D T_{\eta \cd b_1} R_2(\la) \big\|_\infty \leq C \cd (1 + \la)^{-1} \, ,
\end{equation}
for all $\la \geq 0$.

We now use the integral formulae
\[
\begin{split}
F_D & = \frac{1}{\pi} D \cd \int_0^\infty \la^{-1/2} R(\la) \, d\la \\
F_{D_2} & = \frac{1}{\pi} D_2 \cd \int_0^\infty \la^{-1/2} R_2(\la) \, d\la
\end{split}
\]
for the bounded transforms. Indeed, using Lemma \ref{l:fulldomain} one more time, these formulae allow us to compute that
\begin{equation}\label{eq:commutator}
\begin{split}
F_D T_\xi & = F_D T_{\eta \cd b_1} \cd \phi_2(b_2) \\
& = \frac{1}{\pi} D \cd T_{\eta \cd b_1} \cd \int_0^\infty \la^{-1/2} R_2(\la) \cd \phi_2(b_2) \, d\la \\
& \q + \frac{1}{\pi} D \cd \int_0^\infty \la^{-1/2} \big( R(\la) T_{\eta \cd b_1} - T_{\eta \cd b_1} R_2(\la) \big) \cd \phi_2(b_2) \, d\la \\
& = (-1)^{\pa \xi} T_{\eta \cd b_1} F_{D_2} \cd \phi_2(b_2)
+ \frac{1}{\pi} \int_0^\infty \la^{-1/2} L_{\eta \cd b_1} \cd R_2(\la) \cd \phi_2(b_2) \, d\la \\
& \q + \frac{1}{\pi} \int_0^\infty \la^{-1/2} D \cd \big( R(\la) T_{\eta \cd b_1} - T_{\eta \cd b_1} R_2(\la) \big) \cd \phi_2(b_2) \, d\la \, .
\end{split}
\end{equation}
The fact that $D_2 R_2(\la) \phi_2(b_2)$ and $R_2(\la) \phi_2(b_2) \in \mathbb{K}(E_2)$, for all $\la \in [0,\infty)$, combined with the estimate in Equation \eqref{eq:estimate} now imply that both of the integrals on the right hand side of Equation \eqref{eq:commutator} converge absolutely to elements in $\mathbb{K}(E_2,E)$ (remark that the integrands also depend continuously on $\la \in (0,\infty)$ with respect to the operator norm). We thus conclude that
\[
F_D T_\xi - (-1)^{\pa \xi} T_{\eta \cd b_1} F_{D_2} \cd \phi_2(b_2) \in \mathbb{K}(E_2,E) \, .
\]
Since $[F_{D_2}, \phi_2(b_2)] \in \mathbb{K}(E_2)$ we have proved that $F_D T_\xi - (-1)^{\pa \xi} T_\xi F_{D_2} \in \mathbb{K}(E_2,E)$. 

A similar argument shows that $F_D^* T_\xi - (-1)^{\pa \xi} T_\xi F_{D_2}^* \in \mathbb{K}(E_2,E)$ as well.
\end{proof}




\subsection{Localization}
Throughout this subsection the conditions stated in the beginning of Section \ref{s:kuce} are in effect.

We are now going to apply the localization results obtained in Section \ref{s:locaregu}, \ref{s:locahalf} and \ref{s:unbkas}. Recall from Definition \ref{d:localiz} and Proposition \ref{p:loca} that whenever $x \in \sA$, then the localization $D_x : \Tex{Dom}(D_x) \to E_x$ is the selfadjoint and regular unbounded operator defined as the closure of
\[
\phi(x) D \phi(x^*) : \Tex{Dom}(D) \cap E_x \to E_x \, ,
\]
where $E_x := \Tex{cl}\big( \Tex{Im}(\phi(x))\big) \su E$. The core idea is to replace the bounded transform of $D : \Tex{Dom}(D) \to E$ by the bounded transforms of sufficiently many localizations $D_x : \Tex{Dom}(D_x) \to E_x$, when verifying the positivity condition in Equation \eqref{eq:posicond}. The precise result is given here:

\begin{prop}\label{p:lampos}
Suppose that conditions {\em (a)} and {\em (b)} of Definition \ref{defn:loca-subset} hold for some localizing subset $\La \su \sA$ and that $\nu \in \rr$ is given. Suppose moreover that
\[
T_x^* \big( (F_{D_1}^* \hot 1)|_{E_x} \cd F_{D_x} + F_{D_x} \cd (F_{D_1} \hot 1)|_{E_x} \big) T_x + \nu \cd \phi(x^* x)
\]
is positive in $\mathbb{L}(E)/ \mathbb{K}(E)$ for all $x \in \La$. Then we have that
\[
\phi(a^*) \big( (F_{D_1}^* \hot 1) F_D^* + F_D (F_{D_1} \hot 1) \big) \phi(a) + \nu \cd \phi(a^* a)
\]
is positive in $\mathbb{L}(E)/ \mathbb{K}(E)$ for all $a \in A$.
\end{prop}
\begin{proof}
For $x \in \La$ we have that $[F_{D_1} \hot 1, \phi(x) ] = 0$ and the closed submodule $E_x \su E$ is thus invariant under $F_{D_1} \hot 1$. The restriction $(F_{D_1} \hot 1)|_{E_x} : E_x \to E_x$ is therefore a well-defined bounded adjointable operator. The same observation holds for the adjoint $F_{D_1}^* \hot 1$.

Since $\La$ is countable we may write the elements in $\La$ as a sequence $\{x_1,x_2,x_3,\ldots\}$. For each $n \in \nn$, we choose a constant
\[
C_n > 2 + \| x_n \|^2 + \| F_D T_{x_n}^* - T_{x_n}^* F_{D_x} \|_\infty \cd \|x_n\|
\]
and define the element
\[
\Ga := \sum_{n = 1}^\infty \frac{1}{n^2 C_n} x_n^* x_n \, \in A \, ,
\]
where the series is absolutely convergent. Since $\La \cd A \subseteq A$ is norm-dense and $\La = \La^*$ we have that
\[
\Ga \cd A \subseteq A
\]
is norm-dense as well. It therefore suffices to show that
\[
\Ga \cd \big( (F_{D_1}^* \hot 1) F_D^* + F_D (F_{D_1} \hot 1) \big) \cd \Ga + \nu \cd \phi(\Ga^2)
\]
is positive in the Calkin algebra $\mathbb{L}(E)/\mathbb{K}(E)$. 

We now compute modulo $\mathbb{K}(E)$, using Proposition \ref{p:connectionI}, that $\Ga$ commutes with $F_{D_1} \hot 1$ and that $(F_D,E)$ is a Kasparov module:
\[
\begin{split}
& \Ga \cd \big( (F_{D_1}^* \hot 1) F_D^* + F_D (F_{D_1} \hot 1) \big) \cd \Ga \\
& \q \sim \Ga^{1/2} \big( (F_{D_1}^* \hot 1) F_D + F_D (F_{D_1} \hot 1) \big) \cd \Ga^{3/2} \\
& \q = \Ga^{1/2} \sum_{n = 1}^\infty \frac{1}{n^2 C_n} \big( (F_{D_1}^* \hot 1) F_D + F_D (F_{D_1} \hot 1) \big) T_{x_n}^* T_{x_n} \\
& \q \sim
\Ga^{1/2} \sum_{n = 1}^\infty \frac{1}{n^2 C_n} T_{x_n}^* \big( (F_{D_1}^* \hot 1)|_{E_x} F_{D_x} + F_{D_x} (F_{D_1} \hot 1)|_{E_x} \big) T_{x_n} \Ga^{1/2} \, .
\end{split}
\]
But this proves the present proposition since 
\[
\begin{split}
& \Ga^{1/2} \sum_{n = 1}^\infty \frac{1}{n^2 C_n} T_{x_n}^* \big( (F_{D_1}^* \hot 1)|_{E_x} F_{D_x} + F_{D_x} (F_{D_1} \hot 1)|_{E_x} \big) T_{x_n} \Ga^{1/2} \\
& \qqq + \nu \phi(\Ga^2) \\
& \q = 
\Ga^{1/2} \sum_{n = 1}^\infty \frac{1}{n^2 C_n} \Big( T_{x_n}^* \big( (F_{D_1}^* \hot 1)|_{E_x} F_{D_x} + F_{D_x} (F_{D_1} \hot 1)|_{E_x} \big) T_{x_n} \\ 
& \qqq \qq + \nu T_{x_n}^* T_{x_n} \Big)  \cd \Ga^{1/2}
\end{split}
\]
is positive in $\mathbb{L}(E)/\mathbb{K}(E)$ by assumption.
\end{proof}




\subsection{The positivity condition}
We remain in the setup described in the beginning of Section \ref{s:kuce}.

Before continuing our treatment of the positivity condition in Equation \eqref{eq:posicond} we introduce some further notation:

\begin{defn}\label{defn:locaQ}
For each $x \in \sA$ satisfying condition {\em (c)} in Definition \ref{defn:loca-subset} we put
\[
\Tex{Dom}(Q_x) := \Tex{Dom}(D \phi(x^*)) \cap \Tex{Im}(\phi(x))
\]
and define the map $Q_x : \Tex{Dom}(Q_x) \to C$ by
\[
Q_x(\xi) := 2 \cd \Tex{Re} \inn{D \phi(x^*) \xi, (D_1 \hot 1)\phi(x^*) \xi} \, ,
\]
where $\Tex{Re} : C \to C$ takes the real part of an element in the $C^*$-algebra $C$.
\end{defn}


For each $\la \geq 0$ and $x \in \sA$ satisfying condition (b) of Definition \ref{defn:loca-subset} we define the bounded adjointable operators on $E_x$:
\[
\begin{split}
& R_1(\la)|_{E_x} := \big(1 + \la + (D_1^* \hot 1)(D_1 \hot 1)\big)^{-1}|_{E_x} \\
& S_1(\la)|_{E_x} := (D_1 \hot 1) \big(1 + \la + (D_1^* \hot 1)(D_1 \hot 1)\big)^{-1}|_{E_x} \\
& R_x(\la) := (1 + \la + D_x^2)^{-1} \q S_x(\la) := D_x (1 + \la + D_x^2)^{-1} \, . 
\end{split}
\]

The next lemma follows by standard functional calculus arguments:

\begin{lma}\label{l:analysis}
Suppose that $x \in \sA$ satisfies condition {\em (b)} of Definition \ref{defn:loca-subset}. Then the maps $[0,\infty)^2 \to \mathbb{L}(E_x)$ defined by
\[
\begin{split}
M_1(\la,\mu,x) & := S_x(\la) S_1(\mu)|_{E_x} \\
M_2(\la,\mu,x) & := S_x(\la) R_1(\mu)|_{E_x} \cd \sqrt{1+ \mu} \\
M_3(\la,\mu,x)  & := R_x(\la) S_1(\mu)|_{E_x} \cd \sqrt{1 + \la} \\
M_4(\la,\mu,x) & := R_x(\la) R_1(\mu)|_{E_x} \cd \sqrt{(1 + \la)(1 + \mu)}
\end{split}
\]
are all continuous in operator norm and satisfy the estimate
\[
\| M_j(\la,\mu,x) \|_\infty \leq (1 + \la)^{-1/2} \cd (1 + \mu)^{-1/2} \q j \in \{1,2,3,4\} \, ,
\]
for all $\la,\mu \in [0,\infty)$. In particular, it holds that the integral
\[
\frac{1}{\pi^2} \int_0^\infty \int_0^\infty (\la \mu)^{-1/2} \cd (M_j^* M_j)(\la,\mu,x)  \, d\la d\mu
\]
converges absolutely to a bounded adjointable operator $K_j(x) \in \mathbb{L}(E_x)$ with $0 \leq K_j(x) \leq 1$ for all $j \in \{1,2,3,4\}$.
\end{lma}

In order to ensure that later computations are well-defined we prove the following:

\begin{lma}\label{l:domQ}
Suppose that $x \in \sA$ satisfies condition {\em (c)} of Definition \ref{defn:loca-subset} and that $\| d(x^*) \phi(x) \|_\infty < 1$. Then
\begin{equation}\label{eq:incloca}
\Tex{Im}\big( R_x(\la) T_x \big) \su \Tex{Dom}(Q_x) \q \Tex{and} \q
\Tex{Im}\big( S_x(\la) T_x \big) \su \Tex{Dom}(Q_x) \, ,
\end{equation}
for all $\la \geq 0$. In particular, if $x \in \sA$ moreover satisfies condition {\em (b)} of Definition \ref{defn:loca-subset}, then
\[
\Tex{Im}\big( M_j(\la,\mu,x) T_x) \su \Tex{Dom}(Q_x) \, ,
\]
for all $j \in \{1,2,3,4\}$ and all $\la,\mu \in [0,\infty)$.
\end{lma}
\begin{proof}
Recall from Lemma \ref{l:inv} and Proposition \ref{p:loca} that
\[
( i r + D_x)^{-1} T_x = T_x (i r + D \phi(x^* x))^{-1} \, ,
\]
for all $r \in \rr$ with $|r| \geq 1 > \| d(x^*) \phi(x) \|_\infty$. We thus see that
\[
\Tex{Im}\big( (i r + D_x)^{-1} T_x \big) \su \Tex{Im}(\phi(x)) \cap \Tex{Dom}(D \phi(x^*)) = \Tex{Dom}(Q_x) \, .
\]

The inclusions in Equation \eqref{eq:incloca} now follow since
\[
R_x(\la) T_x = (-i \sqrt{1 + \la} + D_x)^{-1} ( i \sqrt{1 + \la} + D_x)^{-1} T_x
\]
and since
\[
S_x(\la) T_x = D_x R_x(\la) T_x = ( i \sqrt{1 + \la} + D_x)^{-1} T_x + i \sqrt{1 + \la} \cd R_x(\la) T_x \, ,
\]
for all $\la \geq 0$.
\end{proof}

We now start a more detailed computation of the application $Q_x : \Tex{Dom}(Q_x) \to C$ from Definition \ref{defn:locaQ}.

\begin{lma}\label{l:prelalg}
Suppose that $x \in \sA$ satisfies condition {\em (b)} and {\em (c)} of Definition \ref{defn:loca-subset} and that $\| d(x^*) \phi(x) \|_\infty < 1$. Then
\[
\begin{split}
Q_x\big( S_x(\la) T_x(\xi)\big) & = 2 \cd \Tex{Re}\binn{ (D_1 \hot 1) \phi(x) \xi, S_x(\la) T_x \xi} \\
& \qq - (1 + \la) Q_x\big( R_x(\la) T_x(\xi)\big) \, ,
\end{split}
\]
for all $\la \in [0,\infty)$ and $\xi \in \Tex{Dom}\big( (D_1 \hot 1)\phi(x) \big)$.
\end{lma}
\begin{proof}
Let $\la \in [0,\infty)$ and let $\xi \in \Tex{Dom}\big( (D_1 \hot 1)\phi(x) \big)$ be given. We first claim that
\[
D T_x^* S_x(\la) T_x \xi \in \Tex{Dom}( (D_1 \hot 1) \phi(x^* x) )
\]
and that
\[
\begin{split}
& (D_1 \hot 1 ) \phi(x^* x) D T_x^* S_x(\la) T_x \xi \\
& \q = (D_1 \hot 1) \phi(x^* x) \xi - (1 + \la) (D_1 \hot 1) T_x^* R_x(\la) T_x \xi \, .
\end{split}
\]
But this follows since
\[
\begin{split}
& \phi(x^* x) D T_x^* S_x(\la) T_x\xi 
= T_x^* D_x S_x(\la) T_x \xi \\
& \q = \phi(x^* x) \xi - (1 + \la) T_x^* R_x(\la) T_x \xi
\in \Tex{Dom}(D_1 \hot 1) \, ,
\end{split}
\]
where we remark that $\phi(x^* x) \xi \in \Tex{Dom}(D_1 \hot 1)$ since $x^* \in \sA$ and that $T_x^* R_x(\la) T_x \xi \in \Tex{Dom}(D) \cap \Tex{Dom}(D_1 \hot 1)$ by condition (c) and Lemma \ref{l:domQ}. 

Notice now that condition (b) and Proposition \ref{p:reg} implies that 
\[
(D_1 \hot 1) \phi(x^* x) : \Tex{Dom}\big( (D_1 \hot 1) \phi(x^* x) \big) \to E
\]
is selfadjoint and regular. Putting $\eta := T_x^* R_x(\la) T_x(\xi) \in \Tex{Dom}(D) \cap \Tex{Dom}(D_1 \hot 1)$ and using the above claim, the lemma is then proved by the following computation:
\[
\begin{split}
& \frac{1}{2} \cd Q_x\big( S_x(\la) T_x(\xi)\big)
= \Tex{Re}\binn{ D T_x^* S_x(\la) T_x(\xi), (D_1 \hot 1) T_x^* S_x(\la) T_x(\xi)} \\
& \q = \Tex{Re}\binn{ (D_1 \hot 1) \phi(x^* x) D T_x^* S_x(\la) T_x \xi, D \eta} \\
& \q = \Tex{Re}\binn{ (D_1 \hot 1) \phi(x^* x) \xi, D \eta} - (1 + \la) \Tex{Re}\binn{ (D_1 \hot 1) \eta, D \eta} \\
& \q = \Tex{Re}\binn{ (D_1 \hot 1) \phi(x) \xi, S_x(\la) T_x \xi} - (1 + \la) \Tex{Re}\binn{ (D_1 \hot 1) \eta, D \eta} \, . 
\end{split}
\]
\end{proof}

\begin{defn}
For each $x \in \sA$ satisfying condition {\em (b)} and {\em (c)} of Definition \ref{defn:loca-subset} and that $\| d(x^*) \phi(x) \|_\infty < 1$, we define the assignment
\[
Q_j(\la,\mu,x) : \Tex{Im}(T_x) \to C \q Q_j(\la,\mu,x)( T_x \xi) := Q_x\big( M_j(\la,\mu,x) T_x \xi \big) \, ,
\]
for all $\la,\mu \in [0,\infty)$, $j \in \{1,2,3,4\}$.
\end{defn}

The main algebraic result of this section can now be stated and proved:

\begin{lma}\label{l:algebra}
Suppose that $x \in \sA$ satisfies condition {\em (b)} and {\em (c)} of Definition \ref{defn:loca-subset} and that $\| d(x^*) \phi(x) \|_\infty < 1$. Then we have the identity
\[
\sum_{j = 1}^4 Q_j(\la,\mu,x)(T_x \xi) = 2 \cd \Tex{Re}\binn{T_x \xi, S_x(\la) S_1(\mu)|_{E_x} T_x \xi} \, ,
\]
for all $\la,\mu \in [0,\infty)$ and all $\xi \in E$.
\end{lma}
\begin{proof}
Let $\la,\mu \in [0,\infty)$ and $\xi \in E$ be given. Remark that $S_1(\mu)\xi \, , \, \, R_1(\mu)\xi \in \Tex{Dom}( (D_1 \hot 1) \phi(x) )$. We may thus use Lemma \ref{l:prelalg} to compute as follows:
\[
\begin{split}
& \sum_{j = 1}^4 Q_j(\la,\mu,x)(T_x \xi) \\
& \q = Q_x( S_x(\la) T_x S_1(\mu) \xi) + Q_x( S_x(\la) T_x R_1(\mu) \xi) (1 + \mu) \\
& \qq + Q_x( R_x(\la) T_x S_1(\mu) \xi) (1 + \la) \\
& \qqq + Q_x( R_x(\la) T_x R_1(\mu) \xi) (1 + \la) (1 + \mu) \\
& \q = 2 \cd \Tex{Re}\binn{ (D_1 \hot 1) \phi(x) S_1(\mu) \xi, S_x(\la) T_x S_1(\mu) \xi } \\
& \qq + 2 \cd \Tex{Re}\binn{ (D_1 \hot 1) \phi(x) R_1(\mu) \xi, S_x(\la) T_x R_1(\mu) \xi } \cd (1 + \mu) \\
& \q = 2 \cd \Tex{Re}\binn{ T_x (D_1^* \hot 1) S_1(\mu) \xi, S_x(\la) T_x S_1(\mu) \xi } \\
& \qq + 2 \cd \Tex{Re}\binn{ T_x S_1(\mu) \xi, S_x(\la) T_x R_1(\mu) \xi } \cd (1 + \mu) \\
& \q = 2 \cd \Tex{Re}\binn{T_x \xi, S_x(\la) S_1(\mu)|_{E_x} T_x \xi} 
\end{split}
\]
This proves the present lemma.
\end{proof}

We are now ready to treat the positivity condition in Equation \eqref{eq:posicond}:

\begin{prop}\label{p:positivity}
Suppose that $\La \subseteq \sA$ is a localizing subset satisfying the local positivity condition, that $\| d(x^*) \phi(x) \|_\infty < 1$ for all $x \in \La$ and that there exists a $\ka > 0$ such that $\ka_x \leq \ka$ for all $x \in \La$. Then the inequality
\[
\phi(a)^* \big( (F_{D_1}^* \hot 1) F_D^* + F_D (F_{D_1} \hot 1)  \big) \phi(a) \geq -  4 \kappa \cd \phi(a^* a)
\]
holds in the quotient $C^*$-algebra $\mathbb{L}(E)/\mathbb{K}(E)$ for all $a \in A$.
\end{prop}
\begin{proof}
By Proposition \ref{p:lampos}, it suffices to show that
\[
T_x^* \big( (F_{D_1}^* \hot 1)|_{E_x} F_{D_x} + F_{D_x} (F_{D_1} \hot 1)|_{E_x} \big) T_x + 4 \kappa \phi(x^* x) 
\]
is positive in $\mathbb{L}(E)/\mathbb{K}(E)$ for all $x \in \La$. Let thus $x \in \La$ be fixed. We will prove the inequality
\[
2 \cd \Tex{Re}\binn{ F_{D_x} (F_{D_1} \hot 1)|_{E_x}  T_x \xi, T_x \xi } 
\geq - 4 \kappa \inn{ T_x \xi, T_x \xi}
\]
in the $C^*$-algebra $C$, for all $\xi \in \Tex{Dom}(D) \cap \Tex{Dom}(D_1 \hot 1)$. Remark that this is enough since $\Tex{Dom}(D) \cap \Tex{Dom}(D_1 \hot 1) \su E$ is norm-dense. 

Let thus $\xi \in \Tex{Dom}(D) \cap \Tex{Dom}(D_1 \hot 1)$ be given. We have that
\[
\begin{split}
& 2 \cd \Tex{Re}\binn{ F_{D_x} (F_{D_1} \hot 1)|_{E_x} \big) T_x \xi, T_x \xi } \\
& \q = 
\frac{2}{\pi^2} \int_0^\infty \int_0^\infty (\la \mu)^{-1/2} \cd \Tex{Re}\binn{S_x(\la) S_1(\mu)|_{E_x} T_x \xi , T_x \xi} \, d\la d\mu \, ,
\end{split}
\]
where the integral converges absolutely in the norm on $C$ and the integrand is norm-continuous from $[0,\infty)^2$ to $C$. Now, by Lemma \ref{l:algebra} and the local positivity condition we have that
\[
\begin{split}
& 2 \cd \Tex{Re}\binn{S_x(\la) S_1(\mu)|_{E_x} T_x \xi , T_x \xi} = \sum_{j = 1}^4 Q_j(\la,\mu,x)(T_x \xi) \\
& \q \geq - \kappa \cd \sum_{j = 1}^4 \inn{M_j(\la,\mu,x) T_x \xi, M_j(\la,\mu,x) T_x \xi} \, .
\end{split}
\]
It therefore follows by Lemma \ref{l:analysis} that
\[
\begin{split}
& \frac{2}{\pi^2} \int_0^\infty \int_0^\infty (\la \mu)^{-1/2} \cd \Tex{Re}\binn{S_x(\la) S_1(\mu)|_{E_x} T_x \xi , T_x \xi} \\
& \q \geq -\kappa \cd \frac{1}{\pi^2}  \int_0^\infty \int_0^\infty (\la \mu)^{-1/2} \\ 
& \qqq  \cd \sum_{j = 1}^4 \inn{M_j(\la,\mu,x) T_x \xi, M_j(\la,\mu,x) T_x \xi} \, d\la d\mu \\
& \q = - \kappa \cd \sum_{j = 1}^4 \inn{T_x \xi , K_j(x) T_x \xi} \geq - 4 \kappa \inn{T_x \xi , T_x \xi} \, .
\end{split}
\]
But this proves the proposition.
\end{proof}

\bibliographystyle{amsalpha-lmp}

\providecommand{\bysame}{\leavevmode\hbox to3em{\hrulefill}\thinspace}
\providecommand{\MR}{\relax\ifhmode\unskip\space\fi MR }
\providecommand{\MRhref}[2]{%
  \href{http://www.ams.org/mathscinet-getitem?mr=#1}{#2}
}
\providecommand{\href}[2]{#2}

\end{document}